\documentclass[10.5pt,times,twoside]{article}

\usepackage{graphicx,latexsym,euscript,makeidx,color,bm}
\usepackage{amsmath,amsfonts,amssymb,amsthm,thmtools,mathtools,mathrsfs,enumerate}
\usepackage[colorlinks,linkcolor=blue,anchorcolor=green,citecolor=red]{hyperref}

\usepackage{geometry}
\def\5n{\negthinspace \negthinspace \negthinspace \negthinspace \negthinspace }
\def\4n{\negthinspace \negthinspace \negthinspace \negthinspace }
\def\3n{\negthinspace \negthinspace \negthinspace }
\def\2n{\negthinspace \negthinspace }
\def\1n{\negthinspace }

\def\dbE{\mathbb{E}}     
\def\dbF{\mathbb{F}} \def\sF{\mathscr{F}}    
         
   \def\cH{{\cal H}}  
   \def\cI{{\cal I}}

 \def\sN{\mathscr{N}}    
     
\def\dbP{\mathbb{P}}     
     
\def\dbR{\mathbb{R}}     
   \def\cS{{\cal S}}

   \def\cV{{\cal V}}

\def\ms{\medskip}         \def\da{\mathop{\downarrow}}

\def\ra{\rightarrow}      
 
\def\no{\noindent}        \def\q{\quad}                      
    \def\qq{\qquad}

            \def\({\Big (}
                  \def\){\Big )}
\def\leq{\leqslant}       \def\geq{\geqslant}
\def\ges{\geqslant}          \def\[{\Big[}
           \def\]{\Big]}
\def\h{\widehat}                   \def\cd{\cdot}
              
\def\ti{\tilde}                            
        \def\ts{\times}                    
\def\pa{\partial}

\def\a{\alpha}        \def\G{\Gamma}      \def\O{\Omega}   
\def\b{\beta}            \def\d{\delta}   \def\F{\Phi}     \def\p{\phi}
               \def\si{\sigma}
\def\e{\varepsilon}   \def\L{\Lambda}  \def\l{\lambda}        
         \def\f{\varphi}     

\def\bde{\begin{definition}\label}    \def\ede{\end{definition}}
\def\be{\begin{equation}}
	\def\bel{\begin{equation}\label}      \def\ee{\end{equation}}
	\def\bt{\begin{theorem}\label}        \def\et{\end{theorem}}
	\def\bc{\begin{corollary}\label}      \def\ec{\end{corollary}}
	\def\bl{\begin{lemma}\label}          \def\el{\end{lemma}}
	\def\bp{\begin{proposition}\label}    \def\ep{\end{proposition}}
	\def\bas{\begin{assumption}\label}    \def\eas{\end{assumption}}
	\def\br{\begin{remark}\label}         \def\er{\end{remark}}
	\def\bex{\begin{example}\label}       \def\ex{\end{example}}
	\def\ba{\begin{array}}                \def\ea{\end{array}}
	\def\ben{\begin{enumerate}}           \def\een{\end{enumerate}}
	
	\newtheorem{theorem}{Theorem}[section]
	\newtheorem{definition}[theorem]{Definition}
	\newtheorem{proposition}[theorem]{Proposition}
	\newtheorem{corollary}[theorem]{Corollary}
	\newtheorem{lemma}[theorem]{Lemma}
	\newtheorem{remark}[theorem]{Remark}
	\newtheorem{example}[theorem]{Example}
	
	\makeatletter
	
	\@addtoreset{equation}{section}
	\makeatother
	
	\allowdisplaybreaks

\begin{document}

	\title{\bf \Large{A Global Maximum Principle for Controlled Conditional Mean-field FBSDEs with Regime Switching}}

		\author{
			Tao Hao\thanks{School of Statistics and Mathematics, Shandong University of Finance and Economics, Jinan 250014, China
                (Email: {\tt taohao@sdufe.edu.cn}).
				This author is supported by Natural Science Foundation of Shandong Province (Grant No. ZR2020MA032),
                National Natural Science Foundation of China (Grant Nos. 11871037, 72171133).}~,~~~
			%
			%
	   Jiaqiang Wen\thanks{Department of Mathematics, Southern University of Science and Technology, Shenzhen 518055, China
                (Email: {\tt wenjq@sustech.edu.cn}.
				This author is supported by National Natural Science Foundation of China (grant No. 12101291) and Guangdong Basic and Applied
                Basic Research Foundation (grant Nos. 2022A1515012017).}~,~~~
			Jie Xiong\thanks{Department of Mathematics and SUSTech International center for Mathematics,
				Southern University of Science and Technology,
				Shenzhen, Guangdong, 518055, China (Email: {\tt xiongj@sustech.edu.cn}).
				This author is supported by SUSTech Start up fund Y01286120 and National Natural Science Foundation of China
                (grant No. 61873325).}
		}
		
		\date{}
		
		\maketitle

		
		\no\bf Abstract. \rm
		This paper is devoted to a global stochastic maximum principle for conditional mean-field forward-backward stochastic differential equations (FBSDEs, for short) with regime switching.
        The control domain is unnecessarily convex and the driver of backward stochastic differential equations (BSDEs, for short) could depend on $Z$.
        Different from the case of non-recursive utility, the first-order and second-order adjoint equations are both high-dimensional
         linear BSDEs.
        Based on the adjoint equations, we reveal the relations among the terms of the first- and second-order Taylor's expansions.
        A general maximum principle is proved, which develops the work of Nguyen, Yin, and Nguyen \cite{Nguyen-Yin-Nguyen-21} to recursive utility.
        As applications, the linear-quadratic problem is considered and a problem with state constraint is studied.

		\ms
		
		\no\bf Key words: \rm Regime switching, conditional mean-field FBSDE, maximum principle, adjoint equation, variational equation.
		
		\ms
		
		\no\bf AMS subject classifications. \rm 60H10, 60H30.

		\section{Introduction}
        Let $T>0$ and $(\O,\sF, \dbF, \dbP)$  be a filtered probability space satisfying the usual conditions, on which a $1$-dimensional standard
Brownian motions $W$ and a continuous-time Markov chain $\a(\cd)$ with a finite state space $\cI=\{1,2,\cd\cd\cd, I\}$ are defined. The generator of
Markov chain $\a(\cd)$ is denoted by $\L=(\l_{ij})_{i,j\in \cI}$, which satisfies $\l_{ij}\ges0$, for $i\neq j\in \cI$ and
$\sum\limits_{j\in\cI}\l_{ij}=0$, for every $i\in \cI$. For each $s>0$, we set
        $$\sF^\a_s=\si\{\a(r): 0\leq r\leq  s\}\vee \sN,\q \sF_s=\si\{W(r),\a(r): 0\leq r\leq  s\}\vee \sN,$$
        where $\sN$ is the set of all $\dbP$-null subsets, and denote $\dbF^\a=(\sF^\a_s)_{s\in[0,T]},\ \dbF=(\sF_s)_{s\in[0,T]}$.

        In this paper, we consider the following forward-backward control system:
        \begin{equation}\label{3.1}
           \hspace{-0.38cm}\left\{\begin{aligned}
               d X^v(t)&=b(t,X^v(t),\dbE[X^v(t)|\sF^{\a}_{t-}],v(t),\a(t-))dt\\
                       &\q+\si(t,X^v(t),\dbE[X^v(t)|\sF^{\a}_{t-}],v(t),\a(t-))dW(t),\ t\in[0,T],\\
               X^{v}(0)&=x,
           \end{aligned}\right.
        \end{equation}
        and
        \begin{equation}\label{3.3}
            \hspace{-0.38cm}\left\{\begin{aligned}
                -d Y^v(t)&=f(t,X^v(t),\dbE[X^v(t)|\sF^{\a}_{t-}],Y^v(t),Z^v(t),v(t),\a(t-))dt-Z^v(t)dW(t),\q t\in[0,T],\\
                 Y^{v}(T)&=\Phi(X^{v}(T),\dbE[X^v(T)|\sF^{\a}_{T-}], \a(T)).
            \end{aligned}\right.
        \end{equation}
        The precise assumptions on the coefficients $b, \sigma,f, \Phi$ refer to Section 3. The cost functional is defined by
        \begin{equation}\label{3.4}
        J(v(\cdot))=Y^v(0).
        \end{equation}
        The purpose of present paper is to give a necessary condition to the optimal control $\bar v(\cd)$, which minimizes the cost functional
$J(v(\cd))$ over $\cV_{0,T}$ (see \autoref{def 3.1}).

        Let us explain the motivation to investigate the equation (\ref{3.1}) for the angle of particle system. Consider an interacting particle
system:
        \begin{equation}\label{1.2}
            \hspace{-0.38cm}\left\{\begin{aligned}
                d X^{i,N}(t)&=b(t,X^{i,N}(t), \frac{1}{N}\sum\limits_{j=1}^NX^{j,N}(t),v^i(t),\a(t-))dt\\
                            &\quad+\si(t,X^{i,N}(t),\frac{1}{N}\sum\limits_{j=1}^NX^{j,N}(t),v^i(t),\a(t-))dW^i(t),\q t\in[0,T],\\
                  X^{i,N}(0)&=x,\q i=1,2,\cd\cd\cd,N,
            \end{aligned}\right.
        \end{equation}
        where $\{W^i;1\leq i\leq N\}$ is $N$ independent standard $1$-dimensional Brownian motions. In many applications, the state process $X^{i,N}$
 is observable while the Brownian motion $W^i$, which is used to model the distribution of $X^{i,N}$, may not be observable. Hence, the closed-loop form of the control $v^i$ is usually used, i.e.,
        $$v^i(t)=\psi(t,X^{i,N}(t), \frac{1}{N}\sum_{j=1}^NX^{j,N}(t),\alpha(t-) ).$$
        Here $\psi$ is a deterministic function depending on the time, the stage, the average state and the switching state. In order to reduce the
computation complexity of the above particle system, generally speaking, let $N\rightarrow\infty$ and consider its limit. As stated  in many
literatures, for example, \cite{Buckdahn-Djehiche-Li-Peng-09, Buckdahn-Li-Peng-09}, according to the law of large numbers, one usually has
        $$\frac{1}{N}\sum_{j=1}^NX^{j,N}(t)\rightarrow \dbE[X^v(t)].$$
        But notice that the appearance of $\alpha(\cdot)$ in all the dynamics of (\ref{1.2}) can  lead to the dynamic of each particle depending on
the history of this process. Consequently, both the mean-field term $\frac{1}{N}\sum_{j=1}^NX^{j,N}$ and its limit (as $N\rightarrow\infty$) depend
on the history of $\a$. Thereby, as $N\rightarrow\infty$, we obtain the conditional mean-field SDE (\ref{3.1}), see \cite{Nguyen-Yin-Hoang-20} for
more details.

        Next, let us overview the history of general maximum principles. As we know, to derive maximum principles, namely, necessary conditions %
for optimality, is an important approach in solving optimal control problems. Its history can be traced back to the work by
Boltyanski-Gamkrelidze-Pontryagin \cite{Boltyanski-Gamkrelidze-Pontryagin-56} in 1956 on the Pontryagin's maximum principle for deterministic control
systems. They introduced the spike variation, and considered the first-order term in a kind of Taylor expansion with respect to this perturbation.
However, if the diffusion term in stochastic control systems depends on control, the approach introduced by
\cite{Boltyanski-Gamkrelidze-Pontryagin-56} does not work. The reason is that the order of $\int_t^{t+\e}\si(t)dW_t$ is
$\sqrt{\e}$ but not $\e$. Later, Peng \cite{Peng-90} considered the second-order term in the Taylor expansion of the variation to solve this
difficulty, and obtained the global maximum principle for  classical stochastic optimal control problems.

       In 1993, Peng \cite{Peng-93} generalized the classical stochastic optimal control problem to one where the cost functional is
defined by $Y(0)$, where $(Y(\cd), Z(\cd))$ is the solution to the following BSDE:
       \begin{equation}\label{3.3-11111}
            \hspace{-0.38cm}\left\{\begin{aligned}
                -d Y^v(t)&=f(t,X^v(t),Y^v(t),Z^v(t),v(t))dt-Z^v(t)dW(t),\q t\in[0,T],\\
                 Y^{v}(T)&= \Phi(X^{v}(T)).
             \end{aligned}\right.
       \end{equation}
The notion of recursive utilities in continuous time is firstly proposed by Duffie, Epstein \cite{Duffie-Epstein-92}. El Karoui, Peng and Quenez \cite{El Karoui-Peng-Quenez-01} developed this notion, and defined a more general class of stochastic recursive utilities in economic theory by solutions of BSDEs.

        When the control domain is convex, one can use the technique of convex variation to obtain a local stochastic maximum principle,
see, for example,  Dokuchaev and Zhou \cite{Dokuchaev-Zhou-99}, Ji and Zhou \cite{Ji-Zhou-06},
Wu \cite{Wu-13}, Xu \cite{Xu-95}.  However, if  the control domain is nonconvex, an essential difficulty is how to construct the first-order and second-order expansions for the BSDE (\ref{3.3-11111}). It is proposed as an open problem in Peng \cite{Peng-98}.

        A method for solving this problem is to regard $Z(\cdot)$ as a control process and the terminal condition $Y(T)=\Phi(X^{v}(T))$ as a
constraint, and apply the Ekeland variational principle to obtain the maximum principle, see Wu \cite{Wu-13}, Yong \cite{Yong-10}. But the maximum principle obtained by this method contains unknown parameters. In 2017, Hu \cite{Hu-17} firstly brought in a new second-order Taylor expansion for the BSDE (\ref{3.3-11111}), and solved the open problem proposed by Peng completely. Hu, Ji and Xue \cite{Hu-Ji-Xue-18} generalized Hu's work to the fully coupled forward-backward stochastic control systems. Hu, Ji and Xu \cite{Hu-Ji-Xu-22}
obtained a global stochastic maximum principle for forward-backward stochastic control systems with quadratic generators.

        As for  stochastic maximum principle for mean-field control systems, we refer to Andersson and Djehiche
\cite{Andersson-Djehiche-11} for a local maximum principle for SDEs of mean-field type, to Buckdahn, Djehiche and Li \cite{Buckdahn-Djehiche-Li-11} for a global stochastic maximum principle for SDEs of mean-field type, to Buckdahn, Li and Ma \cite{Buckdahn-Li-Ma-16}  for a global stochastic maximum principle for general mean-field systems.
Recently, Nguyen, Yin, Hoang \cite{ Nguyen-Yin-Hoang-20} proved the laws of large numbers for systems with mean-field interactions and Markovian switching.
Making use of this result, Nguyen, Nguyen and Yin \cite{Nguyen-Nguyen-Yin-20} obtained a local maximum principle for mean-field type control problems of switching diffusion.
Subsequently, Nguyen, Yin and Nguyen \cite{Nguyen-Yin-Nguyen-21} relaxed the convexity assumption on control domain to the non-convex case.

In this paper, we develop the work of Nguyen, Yin and Nguyen \cite{Nguyen-Yin-Nguyen-21} to  recursive utilities.
It is nontrivial. This is reflected in the following three aspects.
\begin{enumerate}
			\item[(i)]For the non-recursive utilities case, the purpose of constructing first-order adjoint equation
is to use the following equality
         $$
          \dbE[Y^{1,\e}(t)]=\dbE\big[p(t)X^{1,\e}(t)\big],
         $$
where $X^{1,\e}(\cd)$ and $Y^{1,\e}(\cd)$ are the solutions to the first-order variational equations of SDE (\ref{3.1}) and BSDE (\ref{3.3}); $p(\cd)$ is the solution to the
first-order adjoint equation. According to the fact
        $$  \dbE[\dbE[X^{1,\e}(t)|\sF^\a_{t-}]\F(t)]=  \dbE[\dbE[\F(t)|\sF^\a_{t-}]X^{1,\e}(t)],$$
where $\F(t)$ is an $\sF_t$-measurable stochastic process, the first-order adjoint equation can be written as a linear conditional mean-field %
BSDE (see \cite{Nguyen-Yin-Nguyen-21} for example). However, the above approach is not suitable for the
recursive utilities case, since $f$ depends on $(y,z)$. In fact, we need the following slightly ``stronger" relation, $\dbP$-a.s., $t\in[0,T]$,
       \begin{equation}\label{1.0}
       Y^{1,\e}(t)= p_0(t)X^{1,\e}(t)+p_1(t)\dbE[X^{1,\e}(t)|\sF^\a_{t-}].
       \end{equation}
From this, one can know that  the first-order adjoint equation should be a high-dimensional BSDE (without conditional mean-field term)
(see (\ref{5.2-2})-(\ref{5.2-3})).
         \item[(ii)] Due to the appearance of $p_1(\cdot)$ in our method, it comes naturally to deal with the term
$p_1(t)\dbE[X^{2,\e}(t)|\sF^\a_{t-}]$ when deducing the second-order expansion of BSDEs. But the order of the term
$\dbE[(X^{1,\e}(t))^2|\sF^\a_{t-}]$ appearing in $\dbE[X^{2,\e}(t)|\sF^\a_{t-}]$ is $O(\e)$, but not $o(\e)$.
This leads to the order of $p_1(t)\dbE[X^{2,\e}(t)|\sF^\a_{t-}]$ being $O(\e)$, but not $o(\e)$.
Based on the  point above, our second-order adjoint equation is used to deal not only with the resulting impact of $(X^{1,\e}(t))^2$, but
also with that of $\dbE[(X^{1,\e}(t))^2|\sF^\a_{t-}].$
This means that in our case the second-order adjoint equation should also be a high-dimensional BSDE, see (\ref{5.4}).
Note that the first component of  (\ref{5.4}) is just the second-order adjoint equation (4.2) \cite{Nguyen-Yin-Nguyen-21} for the non-recursive utilities case.

\item[(iii)] Since the diffusion term $\sigma$ depends on control, the first- and second-order variational equations for the
BSDE (\ref{3.3}) (see (\ref{4.1-1}), (\ref{4.2-1-1})) involve the term $p^0(t)\d\si(t,v(t))\mathbf{1}_{E_\e}(t)$ in the variation of $z$.
But, since its order is $O(\e)$ for any order expansion of $f$, we need to  consider a new second-order Taylor expansion, and
introduce an auxiliary BSDE (\ref{5.4111}) to handle this obstacle.
		\end{enumerate}

        Compared with the existing literatures, three points should be lighten.
First, since the coefficient $f$ depends on $(y,z)$ and the conditional mean-field term, we establish two pair of new equalities to reveal the
relations among those terms of the first- and second-order Taylor expansions of conditional mean-field FBSDEs,  see (\ref{4.1-2}), (\ref{5.5}).
Second, different to the non-recursive utilities case, the first- and second-order adjoint equations in our case are two high-dimensional linear BSDEs (without conditional mean-field term) (see (\ref{5.2-2}), (\ref{5.4})).
Third, a global maximum principle for conditional mean-field FBSDE (\ref{3.1})-(\ref{3.3}) is proved, which extends the work of Nguyen et al. \cite{Nguyen-Nguyen-Yin-20} from the non-recursive utilities case  to the recursive utilities case (\autoref{th 3.4}).

       This paper is arranged as follows. In Section 2, the first- and second-order adjoint equations as well as the maximum principle are
shown.
The first-order Taylor expansion of the FBSDE (\ref{3.1})-(\ref{3.3}) and some estimates are supplied in Section 3.
The Section 4 is devoted to the second-order Taylor expansion of the FBSDE (\ref{3.1})-(\ref{3.3}).
We study the linear-quadratic case in Section 5.
A problem with state constrain is investigated in Section 6.
In Section 7, some concluding remarks are listed.
In appendix, we supply some proofs.

           \section{Main result}\label{Sec3}

           In this section, we give the main result--stochastic maximum principle.  Throughout this paper, let $V$ be a given nonempty subset of $\dbR$.
           \begin{definition}\label{def 3.1}
           An $\sF_t$-adapted process $v(\cdot)$ with values in $V$ is called an admissible control, if it satisfies
           $$
           \sup_{0\leq t\leq T}\dbE[|v(t)|^8]<\infty.
           $$
           By $\cV_{0,T}$ we denote the set of all admissible controls.
           \end{definition}

          Next, we introduce two spaces which are used frequently: for $\b\geq2,$\\

          \indent $\bullet$\ ${\cS}_{\dbF}^{\b}(0,T;\dbR^n)$ is the family of $\dbR^n$-valued $\dbF$-adapted c\`{a}dl\`{a}g processes
$(\f_t)_{0\leq t\leq T}$ with
          $$\dbE\Big[\mathop{\rm sup}_{0\leq t\leq T}| \f_{t} |^{\b}\Big]< +\infty.$$
          \indent $\bullet$\  $\cH_{\dbF}^{2,\frac{\b}{2}}(0,T;\dbR^{n})$ is the family of $\dbR^n$-valued $\dbF$-progressively measurable
processes $(\f_t)_{0\leq t\leq T}$ with
          $$\dbE\Big[\Big(\int^{T}_{0} |\f_{t}|^{2}dt\Big)^\frac{\b}{2}\Big]<+\infty.$$

          \ms

          Let $ (b,\sigma):[0,T] \times \dbR \times \dbR \times  V  \times \cI\rightarrow \dbR$ satisfy
          \begin{description}
          \item[ ]\textbf{Assumption 1.} {\rm(i)}  There exists a constant $L>0$ such that, for $t\in[0,T],x,x', \bar{x},\bar{x}'\in \dbR,
                       v\in V, i\in \cI$ and for  $\f=b,\sigma,$
               $$\begin{aligned}
                       &|\f(t,x,x',v,i)-\f(t,\bar{x},\bar{x}',v,i)|\leq L(|x-\bar{x}|+|x'-\bar{x}'|),\\
                       &|\f(t,x,x',v,i)|\leq L(1+|x|+|x'|+|v|).
               \end{aligned}$$
	        {\rm(ii)}  The function $\f=b,\si$ is twice continuously differential with respect to $(x,x')$; the derivatives
                       $\f_x,\f_{x'},\f_{xx},\f_{xx'},\f_{x'x'}$ are continuous with respect to $(x,x',v)$, and  are bounded by a constant $L>0$.
           \end{description}

           \ms

           For each $v\in \cV_{0,T}$, under the item (i) of \textbf{Assumption 1} the conditional mean-field SDE (\ref{3.1}) possesses a
unique solution $X^v(\cd)\in \cS^8_{\dbF}(0,T;\dbR).$ Moreover, for any $2\leq \beta\leq 8$, there exists a constant $C>0$ depending on
$L, T, \b$ such that
            \begin{equation}\label{3.2}
            \dbE\Big[\sup_{t\in[0,T]}|X^v(t)|^\b\Big]\leq C\Big(1+\sup_{t\in[0,T]} \dbE\big[|v(t)|^\b\big]\Big).
            \end{equation}
See Lemma 2.4 \cite{Nguyen-Yin-Nguyen-21}.

           \ms

           Let the mappings $f:[0,T] \times \dbR \times \dbR  \times \dbR  \times \dbR  \times V  \times \cI \rightarrow \dbR$ and
$\Phi:\dbR  \times \dbR  \times   \cI\rightarrow \dbR$ satisfy
             \begin{description}
             \item[ ]\textbf{Assumption 2.} {\rm(i)} There exists some constant $L>0$ such that, for $t\in[0,T],
             x,x',\bar{x},\bar{x}'\in\mathbb{R},$ $y,\bar{y}\in \mathbb{R},$ $z,\bar{z}\in \dbR, v, \bar{v}\in V,i\in \cI$,
             $$\begin{aligned}
                 &|f(t,x,x',y,z,v,i)-f(t,\bar{x},\bar{x}',\bar{y},\bar{z},\bar{v},i)|\\
                 &\leq L\Big((1+|x|+|\bar{x}|+|x'|+|\bar{x}'|+|v|+|\bar{v}|)(|x-\bar{x}|+|x'-\bar{x}'|+|v-\bar{v}|)
                 +|y-\bar{y}|+|z-\bar{z}|\Big),\\
                 &|\Phi(x,x',i)-\Phi(\bar{x},\bar{x}',i)|\leq L(1+|x|+|x'|+|\bar{x}|+|\bar{x}'|)(|x-\bar{x}|+|x'-\bar{x}'|),\\
                 &|\Phi(0,0,i)|+|f(t,0,0,0,0,0,0,i)|\leq L.
             \end{aligned}$$
             {\rm(ii)}  The functions $f,\f$ are twice continuously differential with respect to $(x,x',y,z)$ and $(x,x')$, respectively; $Df$ and $D^2f$, the Hessian matrix of $f$  with respect to $(x,x',y,z)$,  are continuous with respect to $(x,x',y,z,v)$;
             $ \F_{x},\F_{x'}, \F_{xx}, \F_{xx'},\F_{x'x'}$ are continuous with respect to $(x,x').$\\
             {\rm(iii)}  The first-order derivatives of $f,  \F$ in $(x,x')$ are bounded by $L(1+|x|+|x'|+|v|)$, $L(1+|x|+|x'|)$, respectively;  $D^2f$ and all the second-order derivatives of $\F$ in $(x,x')$ are  bounded by $L$.
             \end{description}
              \begin{description}
              \item[ ]\textbf{Assumption 3.}
              $b_x(\cdot), \sigma_x(\cdot), b_{xx}(\cdot)$  are  $\dbF^{\alpha}$-adapted.
              \end{description}
\begin{remark} \rm
              \begin{itemize}
			    \item [$\mathrm{(i)}$] \textbf{Assumptions 1}-\textbf{3} covers  the linear-quadratic cases with deterministic
                coefficients and with $\dbF^{\a}$-adapted coefficients.
                \item [$\mathrm{(ii)}$] \textbf{Assumptions 1}-\textbf{3}  covers the case  of a recursive utility and a linear wealth
                (see Duffie and Skiadas \cite{Duffie-Skiadas-94} and Schroder and Skiadas \cite{Schroder-Skiadas-97}), in which the function $f$ does not depend on $z$. In addition,  \textbf{Assumptions 1}-\textbf{3}  also covers the case of the large investor
                (see \cite{El Karoui-Peng-Quenez-01}), in which $f(t,y,c)=-\beta y+u(c)$.
              \end{itemize}
\end{remark}
\begin{lemma}\label{le 3.2}
              Under  \textbf{Assumption 1} and the item $\mathrm{(i)}$ of \textbf{Assumption 2}, the equation (\ref{3.3}) exists a unique solution $(Y^v(\cd),Z^v(\cd))\in \cS^4_{\dbF}(0,T;\dbR) \times \cH^{2,2}_{\dbF}(0,T;\dbR).$ Moreover, for $1< \beta\leq4$, there exists
              a constant $C>0$ depending on $L, T, \b$ such that
              $$
                \begin{aligned}
                    &\mathbb{E}\bigg[\sup\limits_{t\in[0,T]}|Y^v(t)|^\beta+\Big(\int_0^T|Z^v(t)|^2dt\Big)^\frac{\beta}{2}\bigg]
                         \leq C\bigg(1+\sup\limits_{t\in[0,T]} \mathbb{E}\Big[|v(t)|^{2\beta}\Big]\bigg).
                 \end{aligned}
              $$
\end{lemma}
              \begin{proof}
              The proof is immediate from (\ref{3.2}) and \autoref{le 6.2}  in Appendix.
              \end{proof}

             The control $\bar{v}(\cdot)\in \mathcal{V}_{0,T}$ satisfying
             \begin{equation}\label{3.5}
              J(\bar{v}(\cd))=\inf_{v(\cd)\in \cV_{0,T}}J(v(\cd))
             \end{equation}
             is called an optimal control. Let $\bar{X}(\cd):=X^{\bar{v}}(\cd),$ $(\bar{Y}(\cd),\bar{Z}(\cd)):=(Y^{\bar{v}}(\cd),Z^{\bar{v}}(\cd))$
             be the solutions to the equation (\ref{3.1}) and the equation  (\ref{3.3}) with the optimal control $\bar{v}(\cd)$, respectively.
             $(\bar{v}(\cd), \bar{X}(\cd),Y^{\bar{v}}(\cd),Z^{\bar{v}}(\cd))$ is called an optimal pair. The aim of the present paper is to give a necessary condition of the optimal control problem (\ref{3.1})-(\ref{3.3})-(\ref{3.4})-(\ref{3.5}).

            For a stochastic process or a random variable $\xi$, by $\h{\xi}:=\dbE[\xi|\sF^{\a}_{t-}]$  we denote its optimal filtering estimate in the sense of Xiong \cite{Xiong-08}. Denote
           $$p(t)=(p^0(t), p^1(t))^\intercal,\  q(t)=(q^0(t), q^1(t))^\intercal.$$
           Here and thereafter the superscript $\intercal$ denotes the transpose of vectors or matrices.
           Let us consider the following first-order adjoint equation:
           \begin{equation}\label{5.2-2}
           \left\{\begin{aligned}
              dp(t)&=-\Big[F^p(t)p(t)+F^q(t)q(t)+F^f(t)\Big]dt+q(t)dW(t),\\
               p(T)&=F^\Phi(T),
           \end{aligned}\right.
           \end{equation}
           where
           \begin{equation}\label{5.2-3}
           \begin{aligned}
             F^p(t)=& \begin{pmatrix}
                             b_x(t)+f_y(t)+f_z(t)\sigma_x(t)& 0\\
                             b_{x'}(t)+f_z(t)\sigma_{x'}(t)&
                             b_{x}(t)+ \widehat{b}_{x'}(t) +f_y(t)
                       \end{pmatrix}, \\
             F^q(t)=&\begin{pmatrix}
                             \sigma_x(t)+f_z(t)& 0\\
                             \sigma_{x'}(t)& f_z(t)
                      \end{pmatrix}, \qq
             F^{f}(t)=\begin{pmatrix}
                              f_x(t)\\
                              f_{x'}(t)
                      \end{pmatrix}, \qq
             F^{\F}(t)=\begin{pmatrix}
                              \Phi_x(T)\\
                              \Phi_{x'}(T)
                        \end{pmatrix},
           \end{aligned}
           \end{equation}
           which is a $2$-dimensional linear BSDE. Under  \textbf{Assumptions 1}-\textbf{3}, according to \autoref{le 6.1},
           it possesses a unique solution $(p,q)\in \mathcal{S}^2_{\dbF}(0,T;\dbR^2)  \times \cH^{2,1}_{\dbF}(0,T;\dbR^2)$ such that, for $1<\b\leq8$,
           $$
               \dbE\Big[\sup_{t\in[0,T]}|p(t)|^\beta+\Big(\int_0^T|q(t)|^2dt\Big)^\frac{\beta}{2}\Big]<\infty.
           $$
\begin{remark}\label{re 3.1}
			  $\mathrm{(i)}$ \ If $b,\si ,f,\F$ are independent of $x'$, i.e., the case without conditional mean-field term, then
                                     $p_1=q_1\equiv0$. Our first-order adjoint equation reduces to that of classical optimal control problem, see (15) and (25) \cite{Hu-17}.

             $\mathrm{(ii)}$\ If $f$ is independent of $(y,z)$, i.e., the non-recursive utilities case, our first-order adjoint
                                      equation (\ref{5.2-2}) is just the equation (4.1) \cite{Nguyen-Yin-Nguyen-21}. In fact, by setting
                                      $p(t):=p^0(t)+\dbE[p^1(t)|\sF^\a_{t-}]$ and $q(t):=q^0(t)$ one can easily check this argument.
\end{remark}

              By $D_{xyz}^2f$  we denote the Hessian matrix of $f$  with respect to $(x,y,z)$, i.e.,
              $$\begin{aligned}
                 D_{xyz}^2f=\begin{pmatrix}
                              f_{xx}& f_{xy} &f_{xz}\\
                              f_{yx} &f_{yy}&f_{yz}\\
                              f_{zx} &f_{zy}&f_{zz}\\
                            \end{pmatrix},
              \end{aligned}$$
              and denote
             $$\begin{aligned}
                  P(t)&=(P^0(t), P^1(t))^\intercal, \q Q(t)=(Q^0(t), Q^1(t))^\intercal,\q G^\F(T)=(\Phi_{xx}(T), 0)^\intercal.
             \end{aligned}$$
             The second-order adjoint equation is
             \begin{equation}\label{5.4}
                 \left\{\begin{aligned}
                    dP(t)&=-\bigg\{G^P(t)P(t)+G^Q(t)Q(t)+G^p(t)p(t)+G^q(t)q(t)+G^f(t)\bigg\}dt+Q(t)dW(t),\\
                     P(T)&=G^\Phi(T),
                 \end{aligned}\right.
             \end{equation}
             where
             $$
                 \begin{aligned}
                    G^P(t)=&\begin{pmatrix}
                                f_y(t)+2f_z(t)\sigma_x(t)+2b_x(t)+(\sigma_x(t))^2& 0 \\
                                0& f_y(t)+2b_x(t)+(\sigma_x(t))^2 \\
                            \end{pmatrix}, \\
                  G^Q(t)=&\begin{pmatrix}
                              2\sigma_x(t)+f_z(t)& 0   \\
                                                0& f_z(t)\\
                            \end{pmatrix},\\
                    G^{p}(t)=&\begin{pmatrix}
                                b_{xx}(t)+f_z(t)\sigma_{xx}(t)& 0 \\
                                                            0 &b_{xx}(t) \\
                              \end{pmatrix}, \q
                   G^{q}(t)=\begin{pmatrix}
                                \sigma_{xx}(t)& 0\\
                                            0 & 0\\
                             \end{pmatrix}, \\
                    G^f(t)=&\begin{pmatrix}
                              [1,p^0(t),p^0(t)\sigma_x(t)+q^0(t)]D^2_{xyz}f[1,p^0(t),p^0(t)\sigma_x(t)+q^0(t)]^\intercal \\
                              0 \\
                            \end{pmatrix},
              \end{aligned}
          $$
              is  a  $2$-dimensional linear BSDE. Under  \textbf{Assumptions 1}-\textbf{3}, thanks to  \autoref{le 6.1}, it possesses a unique solution $(P,Q)\in \cS^2_{\dbF}(0,T;\dbR^2) \times \cH^{2,1}_{\dbF}(0,T;\dbR^2)$ such that for $1<\b\leq8$,
              $$
              \mathbb{E}\bigg[\sup\limits_{t\in[0,T]}|P(t)|^\beta+\Big(\int_0^T|Q(t)|^2dt\Big)^\frac{\beta}{2}\bigg]<\infty.
              $$

             Define the Hamiltonian associated with random variables $\xi,\xi'\in L^1(\O, \sF,\dbP; \dbR)$, for $(t,y,z,v,i,p^0,p^1,$  $q^0)\in [0,T] \times \dbR \times  \dbR \times V \times \cI \times \dbR \times \dbR \times \dbR$,
             \begin{align} \nonumber
                 \begin{aligned}
                     &H(t,\xi,\xi',y,z,v,i,p^0,p^1,q^0)\\
                     &=p^0b(t,\xi,\xi',v,i)+p^1 \dbE[b(t,\xi,\xi',v,i)|\sF^{\a}_{t-}]+q^0\si(t,\xi,\xi',v,i)\\
                     &\q +f(t,\xi,\xi',y,z+p^0\big[\si(t,\xi,\xi',v,i)-\si(t,\bar{X}(t),\dbE[\bar{X}(t)|\sF^{\a}_{t-}],\bar{v}(t),i)],v,i).
                \end{aligned}
             \end{align}
\begin{theorem}\label{th 3.4}
            Let \textbf{Assumptions 1}-\textbf{3} be in force. Let $\bar{v}(\cd)\in \cV_{0,T}$ be an optimal control, and
            $\bar{X}(\cd)$  $(\bar{Y}(\cd),\bar{Z}(\cd))$ the corresponding state processes of (\ref{3.1}) and (\ref{3.3}) with $\bar{v}(\cd)$, respectively. Then the maximum principle
            \begin{align} \nonumber
                \begin{aligned}
                    &H(t,\bar{X}(t),\dbE[\bar{X}(t)|\sF^{\a}_{t-}],\bar{Y}(t),\bar{Z}(t),v,\a(t-),p^0(t),p^1(t),q^0(t))\\
                    &+\frac{1}{2}P^0(t)\(\si(t,\bar{X}(t),\dbE[\bar{X}(t)|\sF^{\a}_{t-}],v,\a(t-))
                       -\si(t,\bar{X}(t),\dbE[\bar{X}(t)|\sF^{\a}_{t-}],\bar{v}(t),\a(t-))\)^2\\
                   &+\frac{1}{2}P^1(t)  \dbE\bigg[\(\si(t,\bar{X}(t),\dbE[\bar{X}(t)|\sF^{\a}_{t-}],v,\a(t-))
                      -\si(t,\bar{X}(t),\dbE[\bar{X}(t)|\sF^{\a}_{t-}],\bar{v}(t),\a(t-))\)^2\Big|\sF^{\a}_{t-}\bigg]\\
                   &\geq H(t,\bar{X}(t),\dbE[\bar{X}(t)|\sF^{\a}_{t-}],\bar{Y}(t),\bar{Z}(t),\bar{v}(t),\a(t-),p^0(t),p^1(t),q^0(t)),\
                      v\in V,\ \text{a.s.,}\  \text{a.e.}
               \end{aligned}
            \end{align}
            holds true, where $((p^0,p^1), (q^0,q^1))$ and $((P^0,P^1), (Q^0,Q^1))$ are the solutions to the first-
            and second-order adjoint equations (\ref{5.2-2}) and (\ref{5.4}), respectively.
\end{theorem}

            Associate with an optimal seven-tuple $(\bar{X}(\cd),\bar{v}(\cd),p^0(\cd),p^1(\cd),q^0(\cd), P^0(\cd),P^1(\cd))$
            one can define an $\cH$-function
           \begin{align} \nonumber
               \begin{aligned}
                   \cH(t,\xi,\xi',y,z,v,i)&=H(t,\xi,\xi',y,z,v,i,p^0(t),p^1(t),q^0(t))\\
                                          &\q +\frac{1}{2}P^0(t)\(\si(t,\xi,\xi',v,i)-\si(t,\bar{X}(t),\dbE[\bar{X}(t)|\sF^{\a}_{t-}],
                                               \bar{v}(t),i)\)^2\\
                                          &\q+\frac{1}{2}P^1(t)  \dbE\bigg[\(\si(t,\xi,\xi',v,i)-
                                               \si(t,\bar{X}(t),\dbE[\bar{X}(t)|\sF^{\a}_{t-}],\bar{v}(t),i)\)^2\big|\sF^{\a}_{t-}\bigg].
               \end{aligned}
           \end{align}
\begin{corollary}
          We make the same assumption as in \autoref{th 3.4}, then $\mathbb{P}$-a.s., a.e.,
          \begin{align} \nonumber
              \begin{aligned}
                  &\cH(t,\bar{X}(t),\dbE[\bar{X}(t)|\sF^{\a}_{t-}],\bar{Y}(t),\bar{Z}(t),\bar{v}(t),\a(t-))\\
                  &=\min_{v\in V}\cH(t,\bar{X}(t),\dbE[\bar{X}(t)|\sF^{\a}_{t-}],\bar{Y}(t),\bar{Z}(t),v,\a(t-)).
              \end{aligned}
          \end{align}
\end{corollary}
\section{First-order expansion}  \label{Sec4}
         Since the control domain is  unnecessarily convex, we borrow the approach of spike variation to study the variational equations. Precisely,  let $\e>0$ and $E_\e\subset[0,T]$ be a Borel set with Borel measure $|E_\e|=\e$, and  define
         $$v^\varepsilon(t):=
             \left\{\begin{aligned}
                        &\bar{v}(t), \q t\in[0,T] \setminus E_\e,\\
                        & v(t),\quad t\in E_\e.
                   \end{aligned}
              \right.
         $$

         Let $(\bar{v}(\cd), \bar{X}(\cd), \bar{Y}(\cd),\bar{Z}(\cd))$  be  an optimal pair and $X^\e(\cd):=X^{v^\e}(\cd)$,
         $(Y^\e(\cd),Z^\e(\cdot)):=(Y^{v^\e}(\cd),Z^{v^\e}(\cd))$ the solutions to the equations (\ref{3.1}) and (\ref{3.3})  with $v^\e(\cd)$, respectively. Recall that $\h{\xi} :=\dbE[\xi |\sF^{\a}_{t-}]$.

         Set, for $\p=b,\sigma,$
         $$\begin{aligned}
              \delta \p(t,v)&:=\p(t,\bar{X}(t),\h{\bar{X}}(t),v(t),\a(t-))-\p(t,\bar{X}(t),\h{\bar{X}}(t),\bar{v}(t),\a(t-)),\\
                     \p_x(t)&:=\frac{\pa\phi}{\pa x}(t,\bar{X}(t),\h{\bar{X}}(t),\bar{v}(t),\a(t-)),\q
                     \F_{x}(T):=\frac{\pa\Phi}{\pa x}(\bar{X}(T),\h{\bar{X}}(T),\a(T)),\\
                  \p_{xx}(t)&:=\frac{\partial^2\phi}{\partial x^2}(t,\bar{X}(t),\widehat{\bar{X}}(t),\bar{v}(t),\alpha(t-)).
         \end{aligned}$$
        $\p_{x'}(t),\p_{x'x}(t),\phi_{x'x'}(t), \delta  \phi_{x}(t,v), \delta  \phi_{x'}(t,v), \Phi_{x'}(T),\Phi_{xx}(T),\Phi_{xx'}(T),\Phi_{x'x'}(T) $ can be understood similarly.
        For convenience, we denote
        $$\begin{aligned}
              \d^1X(t)&:=X^{\e}(t)-\bar{X}(t),\q   \d^1\h{X}(t):=\h{X}^{\e}(t)-\h{\bar{X}}(t),\\
              \d^1Y(t)&:=Y^{\e}(t)-\bar{Y}(t),\q\  \d^1Z(t):=Z^{\e}(t)-\bar{Z}(t).
        \end{aligned}$$
\begin{lemma}\label{le 4.1} (\cite{Nguyen-Yin-Nguyen-21})
         Under \textbf{Assumptions 1}-\textbf{Assumptions 3}, for any $2\leq \beta\leq 8$, there exists a constant $C_\beta>0$ depending on $\b,T,L$ such that
         $$\begin{aligned}
             \mathrm{i)}\ &\dbE\Big[\sup_{t\in[0,T]}|\d^1\h{X}(t)|^{\b}\Big]\leq C_\b \e^\b,\q
             \mathrm{ii)}\  \dbE\Big[\sup_{t\in[0,T]}|\d^1X(t)|^{\b}\Big]\leq C_\b \e^\frac{\b}{2}.
         \end{aligned}$$
\end{lemma}
        \begin{proof}
        The item $\mathrm{i)}$ is an immediate consequence of Gronwall lemma; the item $\mathrm{ii)}$ comes from  Proposition 3.1 \cite{Nguyen-Yin-Nguyen-21}.
        \end{proof}
\begin{lemma}\label{le 4.1}
         Let \textbf{Assumptions 1}-\textbf{3} hold true, for any $2\leq \beta\leq 4$, there exists a constant $C_\beta>0$ depending on $\b,T,L$ such that
         $$\begin{aligned}
             &\dbE\Big[\sup_{t\in[0,T]}|\d^1Y(t)|^{\b}+\(\int_0^T|\d^1Z(t)|^2dt\)^\frac{\b}{2}\Big]\leq C_\b \e^{\frac{\b}{2}}.
         \end{aligned}$$
\end{lemma}

         \begin{proof}
         Notice that
          $$\begin{aligned}
               \d^1Y(t)&=\F_x^\rho(T)\d^1X(T)+\Phi_{x'}^\rho(T)\delta^1\widehat{X}(T)+\int_t^T\(f_x^{\rho\e}(s)\d^1X(s)
                               +f_{x'}^{\rho\e}(s)\d^1\h{X}(s)\\
                       &\q+f_{y}^{\rho\e}(s)\d^1Y(s)+f_{z}^{\rho\e}(s)\d^1Z(s)+\d f(s,v(s))\mathbf{1}_{E_\e}(s)\)ds-\int_t^T \d^1 Z(s)dW(s),\\
          \end{aligned}$$
          where
          \begin{equation}\label{4.0-1}
              \begin{aligned}
                  \d f(s,v(s))&=f(s,\bar{X}(s),\h{\bar{X}}(s),\bar{Y}(s),\bar{Z}(s),v(s),\a(s-))\\
                              &\qq\qq\qq\qq\qq\qq  -f(s,\bar{X}(s),\h{\bar{X}}(s),\bar{Y}(s),\bar{Z}(s),\bar{v}(s),\a(s-)),\\
               f_x^{\rho\e}(s)&=\int_0^1\pa_x f(s,\bar{X}(s)+\rho(X^\e(s)-\bar{X}(s)),\h{\bar{X}}(s)+\rho(\h{X}^\e(s)-\h{\bar{X}}(s))\\%
                              &\qq\qq \bar{Y}(s)+\rho(Y^\varepsilon(s)-\bar{Y}(s)),\bar{Z}(s)+\rho(Z^\varepsilon(s)-\bar{Z}(s)),v^\e(s),
                                       \a(s-))d\rho,\\
                \F_x^\rho(T)&=\int_0^1\pa_x \F(\bar{X}(T)+\rho(X^\e(T)-\bar{X}(T)), \h{\bar{X}}(T)+\rho(\h{X}^\e(T)-\h{\bar{X}}(T)))d\rho.
             \end{aligned}
          \end{equation}
          $\F_{x'}^\rho(T), f_{x'}^{\rho\e}(s), f_y^{\rho\e}(s),f_z^{\rho\e}(s),\d f_z(s,v(s)),\cd\cd\cd$ can be defined similarly.\\
           According to \autoref{le 6.2}, we have
           \begin{equation}\label{4.1-1-1-1111}
               \begin{aligned}
                  \dbE\[|\d^1Y(t)|^\b+\(\int_0^T|\d^1Z(t)|^2dt\)^\frac{\b}{2}\]
                  \leq\dbE\[|\F_x^\rho(T)\d^1X(T)+\F_{x'}^\rho(T)\d^1\h{X}(T)|^\b \]&\\
                   +\dbE\bigg[\(\int_0^T |f_x^{\rho\e}(s)\d^1X(s)+f_{x'}^{\rho\e}(s)\d^1\h{X}(s)+\d
                         f(s,v(s))\mathbf{1}_{E_\e}(s)|ds\)^\b\bigg].&
               \end{aligned}\end{equation}
           On the one hand, since the first-order derivatives of $\Phi$ are bounded by $L(1+|x|+|x'|)$,  H\"{o}lder inequality and the item $\mathrm{i)}$ allows to show
           \begin{align} \nonumber
               \begin{aligned}
                   &\dbE\[|\F_x^\rho(T)\d^1X(T)+\F_{x'}^\rho(T)\d^1\h{X}(T)|^\b \]\\
                   &\leq C_\b \dbE\bigg[\((1+|\bar{X}(T)|+|X^\e(T)|+|\h{\bar{X}}(T)|+|\h{X}^\e(T)|\)^\beta
                   (|\d^1X(T)|^\b+\d^1\h{X}(T)|^\b)\bigg]\\
                   &\leq C_\b\bigg\{\dbE\[\sup_{s\in[0,T]}(|\d^1X(s)|^{2\b}+|\d^1\widehat{X}(s)|^{2\b})\]\bigg\}^\frac{1}{2}\\
                   &\qq \cd
                   \bigg\{\dbE \[\sup_{s\in[0,T]}(1+|\bar{X}(s)|+|X^\e(s)|+|\h{\bar{X}}(s)|+|\h{X}^\e(s)|)^{2\b}\]\bigg\}^\frac{1}{2}
                   \leq C_\b \e^{\frac{\b}{2}}.
               \end{aligned}
           \end{align}
           One the other hand,  the item $\mathrm{(i)}$ of \textbf{Assumption 2} implies that
           \begin{align} \nonumber
               \begin{aligned}
                   &\dbE\bigg[\(\int_0^T\d f(s,v(s))\mathbf{1}_{E_\e}(s)|ds\)^\b\bigg]\\
                   &\leq \dbE\bigg[\(\int_{E_\e} (1+|\bar{X}(s)|+|\h{\bar{X}}(s)|+|v(s)|+|\bar{v}(s)|)|v(s)-\bar{v}(s)|ds\)^\b\bigg]\\
                   &\leq C_\b \e^{\b-1}\dbE\bigg[\int_{E_\e} (1+|\bar{X}(s)|^{2\b}+|\hat{\bar{X}}(s)|^{2\b}+|v(s)|^{2\b}
                           +|\bar{v}(s)|^{2\b})ds\bigg]\\
                   &\leq C_\b \e^{\b} \sup_{t\in[0,T]}\dbE\Big [1+|\bar{X}(s)|^{2\b}+|\hat{\bar{X}}(s)|^{2\b}+|v(s)|^{2\b}
                           +|\bar{v}(s)|^{2\b}\Big].
               \end{aligned}
           \end{align}
          In addition, since $|f_x^{\rho\e}(s)|+|f_{x'}^{\rho\e}(s)|\leq L(1+|\bar{X}(s)|+|X^\varepsilon(s)|+|{\h X}^\e(s)|
          +|{\widehat{\bar{X}}}(s)|+|v^\e(s)|)$, we derive
           \begin{equation}\label{4.1-1-4}
               \begin{aligned}
               &\dbE\bigg[\(\int_0^T |f_x^{\rho\e}(s)\d^1X(s)+f_{x'}^{\rho\e}(s)\d^1\h{X}(s)|ds\)^\b\bigg]
               \leq C_\b\bigg\{\dbE\[\sup_{t\in[0,T]}(|\d^1X(s)|+|\d^1\h{X}(s)|)^{2\b}\]\bigg\}^\frac{1}{2}\\
               &\qq\cd
               \bigg\{\dbE\[\int_0^T (1+|\bar{X}(s)|^{2\b}+|X^\e(s)|^{2\b}+|{\h X}^\e(s)|^{2\b}
                  +|{\widehat{\bar{X}}}(s)|^{2\b}+|v^\e(s)|^{2\b})ds \]\bigg\}^\frac{1}{2}
              \leq C_\b \e^{\frac{\b}{2}}.
               \end{aligned}
           \end{equation}
           Finally, combining (\ref{4.1-1-1-1111}) and (\ref{4.1-1-4}), we prove the desired result.
           \end{proof}

           Next, we introduce the first-order variational equation
           \begin{equation}\label{4.1}
               \left\{\begin{aligned}
                   d X^{1,\e}(t)&=\(b_x(t)X^{1,\e}(t)+b_{x'}(t)\h{X}^{1,\e}(t)+\d b(t,v(t))\mathbf{1}_{E_\e}(t)\)dt\\
                                &\q +\(\si_x(t)X^{1,\e}(t)+\si_{x'}(t)\h{X}^{1,\e}(t)+\d\sigma(t,v(t))\mathbf{1}_{E_\e}(t)\)dW(t),\\
                     X^{1,\e}(0)&=0,
               \end{aligned}\right.
           \end{equation}
           which is a linear conditional McKean-Vlasov equation (Recall that $\h{X}^{1,\e}(t)=\dbE[X^{1,\e}(t)|\sF^\a_{t-}]$). Under \textbf{Assumption 1}, it possesses a  unique solution $X^{1,\e}(\cd)\in \cS^\b_\dbF(0,T;\dbR)$. Applying Lemma 5.4 \cite{Xiong-08} to (\ref{4.1}) and recall that $b_x(\cd)$ is $\dbF^{\a}$-adapted, we obtain
           \begin{equation}\label{4.1-1-1}
               \left\{\begin{aligned}
                   d \h{X}^{1,\e}(t) &=\bigg\{\( b_x(t)+ \h{b}_{x'}(t)\)\h{X}^{1,\e}(t)+\h{\d b(t,v(t))} \mathbf{1}_{E_\e}(t)\bigg\}dt,\\
                      \h{X}^{1,\e}(0)&=0.
               \end{aligned}\right.
           \end{equation}

          For the solutions to (\ref{4.1}) and (\ref{4.1-1-1}) we have the following moment estimates.
\begin{proposition}\label{pro 4.2}
          Let \textbf{Assumption 1} and \textbf{Assumption 3} be in force. For any $2\leq\b \leq 8$, there exist a constant $C_\b>0$ depending on $\b$
          and a function $\rho:(0,+\infty)\ra(0,+\infty)$ with $\rho(\e)\rightarrow0$ as $\e\da0$ such that
          \begin{equation}\label{4.3}
              \begin{aligned}
                  \mathrm{i)}\q &\dbE\[\sup_{t\in[0,T]}|\h{X}^{1,\e}(t)|^\b\]\leq C_\b\e^\b,\qq
                \mathrm{ii)}\q \dbE\[\sup_{t\in[0,T]}|X^{1,\e}(t)|^\b\]\leq C_\b\e^\frac{\b}{2},\\
                 \mathrm{iii)}\q &  \dbE\[\sup_{t\in[0,T]}|\widehat{X}^{\e}(t)-\widehat{\bar{X}}(t)-\widehat{X}^{1,\e}(t)|^{\b}\]
                 \leq C_\b\e^{\frac{3\b}{2}},\\
                 \mathrm{iv)}\q &   \dbE\[\sup_{t\in[0,T]}|X^{\e}(t)-\bar{X}(t)-X^{1,\e}(t)|^{\b}\]\leq C_\b\e^{\b}.
             \end{aligned}
          \end{equation}
\end{proposition}
         \begin{proof}
          \no  Recall $b_x,\h{b}_{x'}$ are bounded and $\h{\d b(t,v(t))}\leq L\dbE[(1+|\bar{X}(s)|+|\hat{\bar{X}}(s)|+|v(s)|+|\bar{v}(s)|) | \sF_{t-}^\a]$,  the items $\mathrm{i)}$  and $\mathrm{iii)}$ follow from Gronwall lemma. The items $\mathrm{ii)}$ and $\mathrm{iv)}$ refer to Proposition 3.1 \cite{Nguyen-Yin-Nguyen-21}.
          \end{proof}

          We now consider the first-order variational BSDE on $[0,T]$:
          \begin{equation}\label{4.1-1}
              \left\{\begin{aligned}
                   d Y^{1,\e}(t)&=-\bigg\{f_x(t)X^{1,\e}(t)+f_{x'}(t)\h{X}^{1,\e}(t)+f_y(t)Y^{1,\e}(t)
                                   +f_z(t)[Z^{1,\e}(t)-p^0(t)\d\si(t,v(t))\mathbf{1}_{E_\e}(t)]\\
                                &\q-\mathbf{1}_{E_\varepsilon}(t)\Big(q^0(t)\d\si(t,v(t))+p^0(t)\d b(t,v(t))
                                    +p^1(t) \h{\d b(t,v(t))}\)\bigg\}dt+Z^{1,\e}(t)dW(t),\\
                     Y^{1,\e}(T)&=\F_x(T)X^{1,\e}(T)+\F_{x'}(T)\h{X}^{1,\e}(T).
             \end{aligned}\right.
          \end{equation}
          Under \textbf{Assumption 2}, for $2\leq \b \leq 4$ the above linear BSDE exists a unique solution $(Y^{1,\e},Z^{1,\varepsilon})\in \cS^\b_{\dbF}(0,T;\dbR)  \ts  \cH^{2,\frac{\b}{2}}_{\dbF}(0,T;\dbR).$

\ms

          The following lemma reveals the relation of $(Y^{1,\varepsilon}(t),Z^{1,\varepsilon}(t))$  and $X^{1,\varepsilon}(t),\widehat{X}^{1,\varepsilon}(t) $.
\begin{lemma}\label{le 4.3}
          Let \textbf{Assumptions 1}-\textbf{3} be in force, then we have:
          \begin{equation}\label{4.1-2}
              \left\{\begin{aligned}
              &Y^{1,\e}(t)=p^0(t)X^{1,\e}(t)+p^1(t)\h{X}^{1,\e}(t),\\
              &Z^{1,\e}(t)=[p^0(t)\si_x(t)+q^0(t)]X^{1,\e}(t)+[p^0(t)\si_{x'}(t)+q^1(t)]\h{X}^{1,\e}(t)\\
              &\qq\qq      +p^0(t)\delta\sigma(t,v(t))\mathbf{1}_{E_\varepsilon}(t).
             \end{aligned}\right.
          \end{equation}
\end{lemma}
          \begin{proof}
          Applying It\^{o}'s formula to $p^0(t)X^{1,\e}(t)+p^1(t)\h{X}^{1,\e}(t)$, we obtain (\ref{4.1-2}) from the uniqueness of  (\ref{4.1-1}).
          \end{proof}
\begin{remark}\label{4.3-1}
            In the above lemma we use  \textbf{Assumption 3.} As we know, for non-recursive utilities case  (see \cite{Nguyen-Yin-Nguyen-21}),  \textbf{Assumption 3} is not required. Because for non-recursive utilities case, the equality
            \begin{align} \nonumber
            \dbE[Y^{1,\e}(t)]=\dbE\[p^0(t)X^{1,\e}(t)+p^1(t)\h{X}^{1,\e}(t)\]
            \end{align}
            is sufficient for proving the maximum principle. However, for recursive utilities case since the generator $f$ depends on $(y,z)$, we could need the following slightly ``stronger" relation: for $t\in[0,T]$, $\mathbb{P}$-a.s.,
            \begin{align} \nonumber
            Y^{1,\e}(t)=p^0(t)X^{1,\e}(t)+p^1(t)\h{X}^{1,\e}(t).
            \end{align}
             In some way, \textbf{Assumption 3} could be regarded  as the cost of establishing the  ``stronger" relation (\ref{4.1-2}).
\end{remark}

            Denote
            $$\begin{aligned}
                  \delta^2X(t)&:=X^{\e}(t)-\bar{X}(t)-X^{1,\e}(t),\\
                  \delta^2Y(t)&:=Y^{\e}(t)-\bar{Y}(t)-Y^{1,\e}(t),\\
                  \delta^2Z(t)&:=Z^{\e}(t)-\bar{Z}(t)-Z^{1,\e}(t).
            \end{aligned}$$
            Then we have the following estimates.
\begin{proposition}\label{pro 4.5}
           Under \textbf{Assumptions 1}-\textbf{3}, for $2\leq \b\leq 4$, there exist a constant $C_\b>0$ depending on $\b$, and a function $\rho:(0,+\infty)\rightarrow (0,+\infty)$   with  $\rho(\e)\rightarrow0$ as $\e\downarrow0$  such that
           \begin{equation}\label{4.4}
              \begin{aligned}
                  &\mathrm{i)}\quad \dbE\bigg[\sup_{t\in[0,T]}|Y^{1,\e}(t)|^\b+\(\int_0^T|Z^{1,\e}(t)|^2dt\)^{\frac{\b}{2}}\bigg]
                                     \leq C_\b\e^\frac{\b}{2},\\
                 &\mathrm{ii)}\quad \dbE\bigg[\sup_{t\in[0,T]}|\d^2Y(t)|^{4}+\(\int_0^T|\d^2Z(t)|^2dt\)^2\bigg]\leq \e^2 \rho(\e).
              \end{aligned}
           \end{equation}
\end{proposition}
          \no  The proof  refers to Appendix.

\section{Second-order expansion}\label{Sec5}

           This section concerns the second-order expansion of FBSDE (\ref{3.1})-(\ref{3.3}). We consider the following second-order variational equations
           \begin{equation}\label{4.2}
              \left\{\begin{aligned}
                  &d X^{2,\e}(t)=\Big(b_x(t)X^{2,\e}(t)+b_{x'}(t)\h{X}^{2,\e}(t)+\frac{1}{2}b_{xx}(t)(X^{1,\e}(t))^2
                                 +\d b_x(t,v(t))X^{1,\e}(t)\mathbf{1}_{E_\e}(t)\Big)dt\\
                  &\qq\qq\q +\Big(\si_x(t)X^{2,\e}(t)+\si_{x'}(t)\h{X}^{2,\e}(t)+\frac{1}{2}\si_{xx}(t)(X^{1,\e}(t))^2\\
                  &\qq\qq\qq\qq\qq\qq\qq\qq\qq\qq\qq\qq+\d \si_x(t,v(t))X^{1,\e}(t)\mathbf{1}_{E_\e}(t)\Big)dW(t),\\
                  &X^{2,\e}(0)=0,
              \end{aligned}\right.
           \end{equation}
           and
            \begin{equation}\label{4.2-1-1}
               \left\{\begin{aligned}
                      d Y^{2,\e}(t)&= -\bigg\{f_x(t)X^{2,\e}(t)+f_{x'}(t)\h{X}^{2,\e}(t)+f_y(t)Y^{2,\e}(t)+f_z(t)Z^{2,\e}(t)\\
                                   &\q +\frac{1}{2}[1, p^0(t),p^0(t)\si_x(t)+q^0(t)]D_{xyz}^2f(t)
                                      [1, p^0(t),p^0(t)\si_x(t)+q^0(t)]^\intercal(X^{1,\e}(t))^2\\
                         %
                         %
                                   &\q +\mathbf{1}_{E_\e}(t)\(\d f(t,v(t),p^0\d\si(t))+q^0(t)\d\si(t,v(t))+p^0(t)\d b(t,v(t))
                                        +p^1(t) \h{\d b (t,v(t))}\)\bigg\}dt\\
                                   &\q +Z^{2,\varepsilon}(t)dW(s),\\
                        Y^{2,\e}(T)&=\F_x(T)X^{2,\e}(T)+\F_{x'}(T)\h{X}^{2,\e}(T)+\frac{1}{2}\F_{xx}(T)(X^{1,\e}(T))^2,
                \end{aligned}\right.
            \end{equation}
            where

            $$\begin{aligned}
               &\d f(t,v(t),p^0\d\si(t))=f\(t,\bar{X}(t),\dbE[\bar{X}(t)|\sF^\a_{t-}],\bar{Y}(t),\bar{Z}(t)+p^0(t)\d\si(t,v(t)),v(t),\a(t-)\)\\
               &\qq\qq\qq\qq\q -f\(t,\bar{X}(t),\dbE[\bar{X}(t)|\sF^\a_{t-}],\bar{Y}(t),\bar{Z}(t),\bar{v}(t),\a(t-)\).
            \end{aligned}$$
\begin{remark}
            Applying Lemma 5.4 \cite{Xiong-08} to (\ref{4.2}) and recall that $b_x(\cd),b_{xx}(\cd) $ are $\dbF^{\a}$-adapted, we obtain
            \begin{align} \nonumber
               \left\{\begin{aligned}
                   d \h{X}^{2,\e}(t) &=\bigg\{\( b_x(t)+ \h{b}_{x'}(t)\)\h{X}^{2,\e}(t)+\frac{1}{2}b_{xx}(t)\h{(X^{1,\e}(t))^2}
                                       +\h{\d b_x(t,v(t))X^{1,\e}(t)} \mathbf{1}_{E_\e}(t)\bigg\}dt,\\
                \widehat{X}^{2,\e}(0)&=0,
                \end{aligned}\right.
            \end{align}
\end{remark}
            where
           \begin{align} \nonumber
              \begin{aligned}
                  \h{(X^{1,\e}(t))^2}&=\dbE\[(X^{1,\e}(t))^2|\sF_{t-}^\a\], \q
                  \h{\d b_x(t,v(t))X^{1,\e}(t)}&=\dbE\[\d b_x(t,v(t))X^{1,\e}(t)|\sF_{t-}^\a\].
               \end{aligned}
           \end{align}

           Denote
           $$\begin{aligned}
                \d^3X(t)&:=X^{\e}(t)-\bar{X}(t)-X^{1,\e}(t)-X^{2,\e}(t),\\
                \d^3\h{X}(t)&:=\h{X}^{\e}(t)-\h{\bar{X}}(t)-\h{X}^{1,\e}(t)-\h{X}^{2,\e}(t),\\
                \d^3Y(t)&:=Y^{\e}(t)-\bar{Y}(t)-Y^{1,\e}(t)-Y^{2,\e}(t),\\
                \d^3Z(t)&:=Z^{\e}(t)-\bar{Z}(t)-Z^{1,\e}(t)-Z^{2,\e}(t).
           \end{aligned}$$
\begin{lemma}\label{le 5.1}
           Suppose \textbf{Assumption 1} and  \textbf{Assumption 2} hold, for $2\leq \beta\leq4$ there
           exist a constant $C_\b>0$ depending on $\b$ and a function $\rho:(0,+\infty)\rightarrow(0,+\infty)$ with $\rho(\e)\ra0$ as $\e\ra0$ such that
           \begin{align} \nonumber
              \begin{aligned}
              &\mathrm{\mathrm{i)}}\ \dbE\bigg[\sup_{t\in[0,T]}|\widehat{X}^{2,\e}(t)|^\b\bigg]\leq C_\b\e^\frac{3\b}{2},\q
               \mathrm{\mathrm{ii)}}\ \dbE\bigg[\sup_{t\in[0,T]}|X^{2,\e}(t)|^\b\bigg]\leq C_\b\e^\b,\\
               &\mathrm{\mathrm{iii)}}\ \dbE\bigg[\sup_{t\in[0,T]}|\d^3X(t)|^2+\sup_{t\in[0,T]}|\d^3\h{X}(t)|^2\bigg]\leq\e^2 \rho(\e),\\
               & \mathrm{\mathrm{iv)}}\ \dbE\bigg[\sup_{t\in[0,T]}|\d^3Y(t)|^2+\int_0^T|\d^3Z(t)|^2dt\bigg]\leq\e^2\rho(\e).
             \end{aligned}
           \end{align}
\end{lemma}
           \no The item $\mathrm{i)}$ comes from standard estimate for ordinary differential equation and \autoref{pro 4.2};
           the items $\mathrm{ii)}-\mathrm{iii)}$ come from (3.8) and (3.10) \cite{Nguyen-Yin-Nguyen-21}. The proof of item $\mathrm{iv)}$  refers to Appendix.

\ms

           In order to give the relation between   $X^{2,\e}(\cd),\h{X}^{2,\e}(\cd)$ and $(Y^{2,\e}(\cd),Z^{2,\e}(\cd))$, we first
           introduce the following BSDE:
           \begin{equation}\label{5.4111}
              \left\{\begin{aligned}
                   d\tilde{Y}(t)&=-\bigg\{f_y(t)\ti{Y}(t)+f_z(t)\ti{Z}(t)+\mathbf{1}_{E_\e}(t)\bigg(\d f(t,v(t),p^0\d\si(t))+p^0(t)\d
                                    b(t,v(t))\\
                               & \q +p^1(t) \h{\d b(t,v(t))} +q^0(t)\d\si(t,v(t))\\
                               & \q+\frac{1}{2} \(P^0(t)(\d\si(t,v(t))^2 + P^1(t)\h{(\d\si(t,v(t))^2}\)\bigg)\bigg\}dt
                                  -\ti{Z}(t)dW(t),\q t\in[0,T],\\
                  \tilde{Y}(T)&=0.
              \end{aligned}\right.
           \end{equation}
           It is a linear BSDE  and exists a  unique solution $(\ti{Y}(t),\ti{Z}(t))_{t\in[0,T]}\in \cS^2_{\dbF}(0,T;\dbR)\ts
           \cH^{2,1}_{\dbF}(0,T;\dbR).$
\begin{lemma}\label{le 5.2}
           Under \textbf{Assumptions 1}-\textbf{3}, we have
           \begin{equation}\label{5.5}
              \left\{\begin{aligned}
                  Y^{2,\e}(t)&=p^0(t)X^{2,\e}(t)+p^1(t)\h{X}^{2,\e}(t)+\frac{1}{2}P^0(t)(X^{1,\e}(t))^2
                               +\frac{1}{2}P^1(t)\dbE[(X^{1,\e}(t))^2|\sF^\a_{t-}]+\tilde{Y}(t),\\
                 Z^{2,\e}(t)&=[p^0(t)\sigma_x(t)+q^0(t)]X^{2,\varepsilon}(t)+[p^0(t)\si_{x'}(t)+q^1(t)])\h{X}^{2,\e}(t)\\
                    &\q+\frac{1}{2}(X^{1,\e}(t))^2[p^0(t)\si_{xx}(t)+2P^0(t)\si_x(t)+Q^0(t)]+\frac{1}{2}\dbE[(X^{1,\e}(t))^2|\sF^\a_{t-}]Q^1(t)\\
                            &\q+X^{1,\e}(t)\mathbf{1}_{E_\e}(t)\bigg\{P^0(t)\d\si(t,v(t))+p^0(t)\d\si_x(t,v(t)) \bigg\}+\ti{Z}(t).
               \end{aligned}\right.
           \end{equation}
\end{lemma}

           \begin{proof}
            Applying It\^{o}'s formula to
            $$\begin{aligned}
                          &p^0(t)X^{2,\e}(t)+p^1(t)\h{X}^{2,\e}(t)+\frac{1}{2}P^0(t)(X^{1,\e}(t))^2
                               +\frac{1}{2}P^1(t)\dbE[(X^{1,\e}(t))^2|\sF^\a_{t-}] +\tilde{Y}(t),
            \end{aligned}$$
            and making use of the uniqueness of the equation (\ref{4.2-1-1}), we have (\ref{5.5}). The calculation is tedious, but direct.
            \end{proof}
\begin{remark}
            From (\ref{4.1-2}) and (\ref{5.5}), we have, for $t\in[0,T]$, $\mathbb{P}$-a.s.,
            \begin{equation}\label{9.6-3}
               \begin{aligned}
                  Y^\e(t)-\bar{Y}(t)&=p^0(t)\(X^{1,\e}(t)+X^{2,\e}(t)\)+p^1(t)\(\h{X}^{1,\e}(t)+\h{X}^{2,\e}(t)\)\\
                                    &\q +\frac{1}{2}P^0(t)(X^{1,\e}(t))^2+\frac{1}{2}P^1(t)\dbE[(X^{1,\e}(t))^2|\sF^\a_{t-}]+\ti{Y}(t).
               \end{aligned}
            \end{equation}
\end{remark}

            Now we are in a position to give the proof of \autoref{th 3.4}.

            \no \emph{Proof of \autoref{th 3.4}}\ From (\ref{9.6-3}), we have
            \begin{equation}\label{5.7}
            0\leq J(v^\varepsilon(\cdot))-J(\bar{v}(\cdot))=Y^\varepsilon(0)-\bar{Y}(0)=\tilde{Y}(0)+o(\varepsilon).
            \end{equation}
Consider the following SDE:
           \begin{align} \nonumber
           \begin{aligned}
              d\Gamma(t)&=f_y(t)\Gamma(t)dt+f_z(t)\Gamma(t)dW(t),\q
               \Gamma(0)=1.
             \end{aligned}
           \end{align}
           Applying It\^{o}'s formula to $\G(t)\ti{Y}(t)$, integrating from $0$ to $T$ and then taking expectation, it follows
           \begin{align} \nonumber
               \begin{aligned}
                   \ti{Y}(0)&=\dbE\bigg[\int_0^T\G(t) \mathbf{1}_{E_\e}(t)\bigg(\d f(t,v(t),p^0\d\si(t))+p^0(t)\d b(t,v(t))
                              +p^1(t)\dbE[\d b(t,v(t))|\sF^\a_{t-}]\\
                           &\q+q^0(t)\d\si(t,v(t))+\frac{1}{2} \(P^0(t)(\si(t,v(t)))^2 + P^1(t)\dbE[(\si(t,v(t)))^2|\sF^\a_{t-}]\)
                             \bigg)dt\bigg].
              \end{aligned}
           \end{align}
           From the Lebesgue differentiation theorem, we obtain
           $$\begin{aligned}
              &\bigg\{H(t,\bar{X}(t),\dbE[\bar{X}(t)|\sF^{\a}_{t-}],\bar{Y}(t),\bar{Z}(t),v,\a(t-),p^0(t),p^1(t),q^0(t))\\
              &\q+\frac{1}{2}P^0(t)\(\si(t,\bar{X}(t),\dbE[\bar{X}(t)|\sF^{\a}_{t-}],v,\a(t-))-
                  \si(t,\bar{X}(t),\dbE[\bar{X}(t)|\sF^{\a}_{t-}],\bar{v}(t),\a(t-))\)^2\\
              &\q+\frac{1}{2}P^1(t) \dbE\[\(\si(t,\bar{X}(t),\dbE[\bar{X}(t)|\sF^{\a}_{t-}],v,\a(t-))-
                  \si(t,\bar{X}(t),\dbE[\bar{X}(t)|\sF^{\a}_{t-}],\bar{v}(t),\a(t-))\)^2\big|\sF^{\a}_{t-}\]\bigg\}\G(t)\\
              &\geq H(t,\bar{X}(t),\dbE[\bar{X}(t)|\sF^{\a}_{t-}],\bar{Y}(t),\bar{Z}(t),\bar{v}(t),\a(t-),p^0(t),p^1(t),q^0(t))
              \Gamma(t),\   v\in V,\ \text{a.s.,}\  \text{a.e.}
           \end{aligned}$$
Since $\Gamma(t)>0,\ t\in[0,T]$, we get the wished result. $\square$

\section{Linear-quadratic case}

In this section we focus on the linear-quadratic case.
Recall that $\widehat{\xi}:=\mathbb{E}[\xi|\sF^{\a}_{t-}]$  denotes
its optimal filtering estimate.
Consider the following linear forward-backward control system
\begin{equation}\label{10.1}
\left\{
\begin{aligned}
d X^v(t)&=\Big(A_1(t,\alpha(t-))X^v(t)+A_2(t,\alpha(t-))\widehat{X}^v(t)+A_3(t,\alpha(t-))v(t)\Big)dt\\
         &\quad+\Big(B_1(t,\alpha(t-))X^v(t)+B_2(t,\alpha(t-))\widehat{X}^v(t)+B_3(t,\alpha(t-))v(t)\Big)dW(t),\ t\in[0,T],\\
         X^{v}(0)&=x,
\end{aligned}
\right.
\end{equation}
\begin{equation}\label{10.2}
\left\{
\begin{aligned}
d Y^v(t)&=-\bigg(C_1(t,\alpha(t-))X^v(t)+C_2(t,\alpha(t-))\widehat{X}^v(t)+ C_3(t,\alpha(t-))Y^v(t)+C_4(t,\alpha(t-))Z^v(t)\\
         &\quad  +C_5(t,\alpha(t-))v(t)\bigg)dt+Z^v(t)dW(t),\ t\in[0,T],\\
         Y^{v}(T)&=D_1(\alpha(T))(X^v(T))^2+D_2(\alpha(T))(\widehat{X}^v(T))^2,
\end{aligned}
\right.
\end{equation}
where $A_1(\cdot),A_2(\cdot),\cdot\cdot\cdot,D_2(\cdot)$ are $\mathbb{F}^\alpha$-adapted bounded processes.
For convenience, we write $A_1(t,\alpha(t-)),$  $A_2(t,\alpha(t-))\cdot\cdot\cdot D_2(\alpha(T)) $ as $A_1,A_2 \cdot\cdot\cdot D_2$.
It is easy to check that the  coefficients in (\ref{10.1}) and (\ref{10.2}) satisfy \textbf{Assumptions 1}-\textbf{3}.

In this setting, the first-order adjoint equation becomes
\begin{equation}\label{10.3}
\left\{
\begin{aligned}
d
\begin{pmatrix}
p^{0}(t)\\
p^{1}(t)
\end{pmatrix}
&=-\bigg[\begin{pmatrix}
A_1+C_3+B_1C_4& 0\\
A_2+B_2C_4    &  A_1+A_2+C_3
\end{pmatrix}
\begin{pmatrix}
p^{0}(t)\\
p^{1}(t)
\end{pmatrix}
+\begin{pmatrix}
B_1+C_4& 0\\
B_2    & C_4
\end{pmatrix}\begin{pmatrix}
q^{0}(t)\\
q^{1}(t)
\end{pmatrix}\\
&\quad+
\begin{pmatrix}
C_1\\
C_2
\end{pmatrix}\bigg]dt
+\begin{pmatrix}
q^{0}(t)\\
q^{1}(t)
\end{pmatrix}
dW(t),\\
\begin{pmatrix}
p^{0}(T)\\
p^{1}(T)
\end{pmatrix}&=\begin{pmatrix}
2D_1X(T)\\
2D_2\widehat{X}(T)
\end{pmatrix},
\end{aligned}
\right.
\end{equation}
and the second-order adjoint equation is
\begin{equation}\label{10.4}
\left\{
\begin{aligned}
d P^0(t)&=-\bigg\{\Big(C_3+2B_1C_4+2A_1+(B_1)^2\Big)P^0(t)+\Big(2B_1+C_4\Big)Q^0(t)\bigg\}dt\\
        &\quad +Q^0(t)dW(t),\ t\in[0,T],\\
P^0(T)&= 2D_1,
\end{aligned}
\right.
\end{equation}
with $P^1(t)=Q^1(t)\equiv0,\ t\in[0,T]$.

From \autoref{th 3.4}, we have
\begin{theorem}
Let $\bar{v}$ be the optimal control, and let $((p^0(\cdot),p^1(\cdot)), (q^0(\cdot),q^1(\cdot)))$ and
$(P^0(\cdot),Q^0(\cdot))$ be the solutions to the first-order adjoint equation (\ref{10.3}) and the
second-order adjoint equation (\ref{10.4}), respectively. Then the following maximum principle holds true
\begin{equation}\label{10.5}
\begin{aligned}
&\Big[ p^0(t)(A_3+C_4B_3)+B_3q_0(t)+C_5\Big](v-\bar{v}(t))+p^1(t)A_3\mathbb{E}[(v-\bar{v}(t))|\sF^\alpha_{t-}]\\
&+\frac{1}{2}P^0(t)(B_3)^2(v-\bar{v}(t))^2\geq0,\q v\in V,\ \text{a.s.,}\  \text{a.e.}
\end{aligned}
\end{equation}
\end{theorem}

\section{Application to problems with state constraint}

As an illustrative application, in this section we investigate the
corresponding problem with state constraint. More precisely, we consider
the forward-backward control system (\ref{3.1})-(\ref{3.3}), the cost functional (\ref{3.4})
as well as the state constrain:
\begin{equation}\label{9.1}
\mathbb{E}[\Psi(X(T), \widehat{X}(T), Y(0))]=0,
\end{equation}
where $\Psi:\mathbb{R}\times \mathbb{R}\times \mathbb{R}\rightarrow \mathbb{R}$.
\begin{description}
             \item[ ]\textbf{Assumption 4.}
             The function $\Psi$ is twice continuously differentiable with respect to $(x,x',y)$, and the Hessian matrix of $\Psi$  with respect to $(x,x',y)$, denoted by $D^2\Psi$, is bounded.
 \end{description}

For any $v(\cdot)\in \mathcal{V}_{0,T}$, we denote $X^v(\cdot), (Y^v(\cdot),Z^v(\cdot))$ the solutions to the equation (\ref{3.1})
and the equation (\ref{3.3}), respectively. The set of admissible controls is defined by
\begin{align} \nonumber
\mathbb{V}_{0,T}=\Big\{v(\cdot)\in \mathcal{V}_{0,T}\Big| \mathbb{E}[\Psi(X^v(T), \widehat{X}^v(T), Y^v(0))]=0 \Big\}.
\end{align}

Let $\bar{v}(\cdot)\in \mathbb{V}_{0,T}$ be the optimal control and $\bar{X}(\cdot), (\bar{Y}(\cdot),\bar{Z}(\cdot))$
the optimal trajectories. For arbitrary constant $\kappa>0$, we consider the following cost functional on $\mathcal{V}_{0,T}$:
\begin{align} \nonumber
J_\kappa(v(\cdot))=\bigg\{\Big(Y^v(0)-\bar{Y}(0)+\kappa  \Big)^2+\Big|\mathbb{E}[\Psi(X^v(T), \widehat{X}^v(T), Y^v(0))]\Big|^2\bigg\}^\frac{1}{2}.
\end{align}
Clearly, $$J_\kappa(v(\cdot))>0\q \text{and}\q J_\kappa(\bar{v}(\cdot))=\kappa \leq \inf\limits_{v\in\mathcal{V}_{0,T}}J_\kappa(v(\cdot))+\kappa.$$

Next, we use Ekeland's variational principle to study stochastic maximum principle. For this, let us
introduce the metric on $\mathcal{V}_{0,T}$:
$$d(u(\cdot), v(\cdot))=\mathbb{E}\Big[\int_0^T\mathbf{1}_{\{u\neq v\}}(t,\omega) dt  \Big].$$
Let $(\mathcal{V}_{0,T},d)$ be a complete space, otherwise, one can adopt the argument in Tang and Li \cite{Tang-Li-94}, Wu \cite{Wu-13} to derive the same result.
According to Ekeland's variational principle, there exists a $v_\kappa(\cdot)\in \mathcal{V}_{0,T}$ such that
\begin{equation}\label{9.4}
\begin{aligned}
&\mathrm{i)}\ J_\kappa(v_\kappa(\cdot))\leq J_\kappa(\bar{v}(\cdot))=\kappa,\\
&\mathrm{ii)}\ d(v_\kappa(\cdot), \bar{v}(\cdot) )\leq\sqrt{\kappa},\\
&\mathrm{iii)}\ J_\kappa(v(\cdot))-J_\kappa(v_\kappa(\cdot))+\sqrt{\kappa}d(v(\cdot), v_\kappa(\cdot))\geq0,\q \forall v(\cdot)\in \mathcal{V}_{0,T}.
\end{aligned}
\end{equation}
For arbitrary $\varepsilon>0$ and $v\in \mathcal{V}_{0,T}$, define
$$
v_{\kappa,\varepsilon}(t)=v_{\kappa}(t)\mathbf{1}_{E^c_\varepsilon}(t)+v(t)\mathbf{1}_{E_\varepsilon}(t),
$$
where $E_\varepsilon\subset[0,T]$ is a Borel subset with its Borel measure $|E_\varepsilon|=\varepsilon$.
Clearly, $d(v_{\kappa,\varepsilon}(t), v_{\kappa}(t))\leq \varepsilon.$
By $(X^\kappa(\cdot), Y^\kappa(\cdot),Z^\kappa(\cdot))$ and $(X^\varepsilon(\cdot), Y^\varepsilon(\cdot),Z^\varepsilon(\cdot))$
we denote the solutions to the forward-backward control system (\ref{3.1})-(\ref{3.3}) with $v_\kappa(\cdot)$ and $v_{\kappa,\varepsilon}(\cdot)$, respectively. From the item $\mathrm{iii)}$ of (\ref{9.4}) and Taylor expansion, we have
\begin{equation}\label{6.5-111}
\begin{aligned}
0&\leq  J_\kappa(v_{\kappa,\varepsilon}(\cdot))-J_\kappa(v_\kappa(\cdot))+\sqrt{\kappa}\varepsilon\\
 &\leq \lambda_\kappa(Y^\varepsilon(0)-Y^\kappa(0))
 +\mu_\kappa\bigg\{\mathbb{E}\Big[ \Psi(X^\varepsilon(T), \widehat{X}^\varepsilon(T), Y^\varepsilon(0))\Big]-
 \mathbb{E}\Big[ \Psi(X^\kappa(T), \widehat{X}^\kappa(T), Y^\kappa(0))\Big] \bigg\}\\
 &\quad+\sqrt{\kappa}\varepsilon+o(\varepsilon),
\end{aligned}
\end{equation}
where
\begin{equation}\label{9.6-2}
\lambda_\kappa=\frac{1}{J_\kappa(v_\kappa(\cdot))}[Y^\kappa(0)-\bar{Y}(0)+\kappa],\qq
\mu_\kappa= \frac{1}{J_\kappa(v_\kappa(\cdot))}\mathbb{E}\Big[ \Psi(X^\kappa(T), \widehat{X}^\kappa(T), Y^\kappa(0))\Big].
\end{equation}

By
$$
\begin{aligned}
(p^\kappa, q^\kappa)=\Big((p^{0,\kappa},p^{1,\kappa}), (q^{0,\kappa},q^{1,\kappa})\Big),\
(P^\kappa,Q^\kappa)=\Big((P^{0,\kappa},P^{1,\kappa}), (Q^{0,\kappa},Q^{1,\kappa})\Big)
\end{aligned}
$$
we denote the solutions to the first-order adjoint equation (\ref{5.2-2}) and the second-order adjoint
equation (\ref{5.4}) but with $(X^\kappa(\cdot), Y^\kappa(\cdot),Z^\kappa(\cdot),v^\kappa(\cdot))$ instead of
$(\bar{X}(\cdot), \bar{Y}(\cdot),\bar{Z}(\cdot),\bar{v}(\cdot))$.

On the one hand, by making the similar analyse as (\ref{5.7}), we can obtain
\begin{equation}\label{9.6}
Y^\varepsilon(0)-Y^\kappa(0)=\tilde{Y}^\kappa(0)+o(\varepsilon),
\end{equation}
where
\begin{align} \nonumber
\begin{aligned}
\tilde{Y}^\kappa(t)&=\int_t^T\bigg\{f^\kappa_y(t)\tilde{Y}^\kappa(t)+f^\kappa_z(t)\tilde{Z}^\kappa(t)
+\mathbf{1}_{E_\varepsilon}(t)\bigg(\delta f^\kappa(t,v(t),p^{0,\kappa}\delta\sigma^\kappa(t))+p^{0,\kappa}(t)\delta b^\kappa(t,v(t))\\
                &\quad+p^{1,\kappa}(t)\mathbb{E}[\delta b^\kappa(t,v(t))|\mathcal{F}^\alpha_{t-}]+q^{0,\kappa}(t)\delta\sigma^\kappa(t,v(t))\\
               &\q+\frac{1}{2} \Big(P^{0,\kappa}(t)(\delta\sigma^\kappa(t,v(t)))^2
                    +P^{1,\kappa}(t)\mathbb{E}[(\delta\sigma^\kappa(t,v(t)))^2|\sF^\alpha_{t-}]\Big)\bigg)\bigg\}dt
               -\int_t^T\tilde{Z}^\kappa(t)dW(t)
\end{aligned}
\end{align}
and $f_y^\kappa(t)=\partial_yf(t,X^\kappa(t),\widehat{X}^\kappa(t),Y^\kappa(t),Z^\kappa(t),v^k(t),\alpha(t-))$.
$f_z^\kappa(t), \delta f^\kappa(t),\cdot\cdot\cdot$ can be understood similarly.

On the other hand,
inspired by the previous maximum principle in Section 3, we consider the following two BSDEs:

\begin{align}
\left\{
\begin{aligned}
d
\begin{pmatrix}
\widetilde{p}^{0,\kappa}(t)\\
\widetilde{p}^{1,\kappa}(t)
\end{pmatrix}
&=-\bigg[\begin{pmatrix}
b^\kappa_x(t)& 0\\
b^\kappa_{x'}(t)&
b^\kappa_{x}(t)+ \widehat{b^\kappa_{x'}}(t)
\end{pmatrix}
\begin{pmatrix}
\widetilde{p}^{0,\kappa}(t)\\
\widetilde{p}^{1,\kappa}(t)
\end{pmatrix}
+\begin{pmatrix}
\sigma^\kappa_x(t)& 0\\
\sigma^\kappa_{x'}(t)& 0
\end{pmatrix}\begin{pmatrix}
\widetilde{q}^{0,\kappa}(t)\\
\widetilde{q}^{1,\kappa}(t)
\end{pmatrix}\bigg]dt\\
&\quad
+\begin{pmatrix}
\widetilde{q}^{0,\kappa}(t)\\
\widetilde{q}^{1,\kappa}(t)
\end{pmatrix}
dW(t),\q t\in[0,T],\\
\widetilde{p}(T)&=\begin{pmatrix}
\Phi^\kappa_x(T)\\
\Phi^\kappa_{x'}(T)
\end{pmatrix},
\end{aligned}
\right.
\end{align}

and
\begin{align}
\left\{
\begin{aligned}
d\widetilde{P}^\kappa(t)&=-\bigg\{G^{\widetilde{P}^\kappa}(t)\widetilde{P}^\kappa(t)
+G^{\widetilde{Q}^\kappa}(t)\widetilde{Q}^\kappa(t)+G^{\tilde{p}^\kappa}(t)\tilde{p}^\kappa(t)
+G^{\widetilde{q}^\kappa}(t)\widetilde{q}^\kappa(t)\bigg\}dt+\widetilde{Q}^\kappa(t)dW(t),\\
\widetilde{P}^\kappa(T)&=G^\Phi(T),
\end{aligned}
\right.
\end{align}
where
\begin{equation}\label{5.4-1-1-1-3}
\begin{aligned}
\widetilde{P}^\kappa(t)=&(\widetilde{P}^{0,\kappa}(t),\widetilde{P}^{1,\kappa}(t))^\intercal,\q \widetilde{Q}^\kappa(t)=(\widetilde{Q}^{0,\kappa}(t),\widetilde{Q}^{1,\kappa}(t))^\intercal,\\
\widetilde{p}^\kappa(t)=&(\widetilde{p}^{0,\kappa}(t),\widetilde{p}^{1,\kappa}(t))^\intercal,\q \widetilde{q}^\kappa(t)=(\widetilde{q}^{0,\kappa}(t),\widetilde{q}^{1,\kappa}(t))^\intercal,\\
G^{\widetilde{P}^\kappa}(t)=&
\begin{pmatrix}
2b^\kappa_x(t)+(\sigma^\kappa_x(t))^2 &0 \\
0& 2b^\kappa_x(t)+(\sigma^\kappa_x(t))^2   \\
%
\end{pmatrix},  \q
G^{\widetilde{Q}^\kappa}(t)=
\begin{pmatrix}
2\sigma^\kappa_x(t)& 0      \\
0& 0        \\
%
\end{pmatrix}, \\
G^{\widetilde{p}^\kappa}(t)=&
\begin{pmatrix}
b^\kappa_{xx}(t)& 0\\
0 &b^\kappa_{xx}(t)\\
%
\end{pmatrix},\q
G^{\widetilde{q}^\kappa}(t)=
\begin{pmatrix}
\sigma^\kappa_{xx}(t) & 0\\
0 &0\\
%
\end{pmatrix}, \q
G^\Phi(T)=
\begin{pmatrix}
\Phi^\kappa_{xx}(T)\\
0\\
%
\end{pmatrix}.
\end{aligned}
\end{equation}
According to (\ref{9.6})-(\ref{5.4-1-1-1-3}), similar to the proof of the previous maximum principle we can deduce
\begin{align} \nonumber
\begin{aligned}
&\mu_\kappa\bigg\{\mathbb{E}\Big[ \Psi(X^\varepsilon(T), \widehat{X}^\varepsilon(T), Y^\varepsilon(0))\Big]-
 \mathbb{E}\Big[ \Psi(X^\kappa(T), \widehat{X}^\kappa(T), Y^\kappa(0))\Big] \bigg\}\\
 &=\mathbb{E}\bigg[\int_0^T\mathbf{1}_{E_\varepsilon}(t)\bigg(\widetilde{p}^{0,\kappa}(t)\delta b^\kappa(t,v(t))+\widetilde{p}^{1,\kappa}(t)\mathbb{E}[\delta b^\kappa(t,v(t))|\sF^\alpha_{t-}]+\widetilde{q}^{0,\kappa}(t)\delta\sigma^\kappa(t,v(t))\\
 &\quad+\frac{1}{2} \Big(\widetilde{P}^{0,\kappa}(t)(\delta\sigma^\kappa(t,v(t)))^2 + \widetilde{P}^{1,\kappa}(t)\mathbb{E}[(\delta\sigma^\kappa(t,v(t)))^2|\sF^\alpha_{t-}]\Big)\bigg)dt\bigg]\\
 &\quad +\mu_\kappa\mathbb{E}\Big[\Psi_y(X^\kappa(T), \widehat{X}^\kappa(T), Y^\kappa(0))\Big]\tilde{Y}^\kappa(0)+o(\varepsilon).
  \end{aligned}
\end{align}
Consequently, (\ref{6.5-111}) can be written as
\begin{equation}\label{6.5-1113333}
\begin{aligned}
0&\leq  J_\kappa(v_{\kappa,\varepsilon}(\cdot))-J_\kappa(v_\kappa(\cdot))+\sqrt{\kappa}\varepsilon\\
 &\leq \Big(\lambda_\kappa+\mu_\kappa\mathbb{E}\Big[\Psi_y(X^\kappa(T), \widehat{X}^\kappa(T), Y^\kappa(0))\Big]\Big)\tilde{Y}^\kappa(0)    \\
 &\quad+\mathbb{E}\bigg[\int_0^T\mathbf{1}_{E_\varepsilon}(t)\bigg(\widetilde{p}^{0,\kappa}(t)\delta b^\kappa(t,v(t))+\widetilde{p}^{1,\kappa}(t)\mathbb{E}[\delta b^\kappa(t,v(t))|\sF^\alpha_{t-}]+\widetilde{q}^{0,\kappa}(t)\delta\sigma^\kappa(t,v(t))\\
 &\quad+\frac{1}{2} \Big(\widetilde{P}^{0,\kappa}(t)(\delta\sigma^\kappa(t,v(t)))^2 + \widetilde{P}^{1,\kappa}(t)\mathbb{E}[(\delta\sigma^\kappa(t,v(t)))^2|\sF^\alpha_{t-}]\Big)\bigg)dt\bigg]
    +\sqrt{\kappa}\varepsilon+o(\varepsilon).
\end{aligned}
\end{equation}

Next, let us consider the following SDE:
\begin{align} \nonumber
\left\{
\begin{aligned}
d\Upsilon^\kappa(t)&=f_y^\kappa(t)\Upsilon^\kappa(t)dt+f_z^\kappa(t)\Upsilon^\kappa(t)dW(t),\\
\Upsilon^\kappa(0)&=\lambda_\kappa+\mu_\kappa\mathbb{E}\Big[\Psi_y(X^\kappa(T), \widehat{X}^\kappa(T), Y^\kappa(0))\Big].
\end{aligned}
\right.
\end{align}
Then, applying It\^{o}'s formula to $\Upsilon^\kappa(t)\tilde{Y}^\kappa(t)$ one has
\begin{equation}\label{9.15}
\begin{aligned}
&\Big(\lambda_\kappa+\mu_\kappa\mathbb{E}\Big[\Psi_y(X^\kappa(T), \widehat{X}^\kappa(T), Y^\kappa(0))\Big]\Big)\tilde{Y}^\kappa(0) \\
&=\mathbb{E}\bigg[\int_0^T\Upsilon^\kappa(t)\mathbf{1}_{E_\varepsilon}(t)\bigg\{p^{0,\kappa}(t)\delta b^\kappa(t,v(t))
                +p^{1,\kappa}(t)\mathbb{E}[\delta b^\kappa(t,v(t))|\sF^\alpha_{t-}]+q^{0,\kappa}(t)\delta\sigma^\kappa(t,v(t))\\
                &\quad+\frac{1}{2} \Big(P^{0,\kappa}(t)(\delta\sigma^\kappa(t,v(t)))^2 + P^{1,\kappa}(t)\mathbb{E}[(\delta\sigma^\kappa(t,v(t)))^2|\sF^\alpha_{t-}]\Big)\\
          &\quad+f\Big(t,X^\kappa(t),\widehat{ X}^\kappa(t) ,Y^\kappa(t),Z^\kappa(t)+p^{0,\kappa}(t)
          \delta\sigma^\kappa(t,v(t)),v(t),\alpha(t-)\Big)\\
&\qq\qq\qq\qq\qq \qq\qq\qq\qq
-f\Big(t,X^\kappa(t),\widehat{X}^\kappa(t),Y^\kappa(t),Z^\kappa(t),v^\kappa(t),\alpha(t-)\Big)\bigg\}dt\bigg].
\end{aligned}
\end{equation}

Define the Hamiltonian, for $\xi,\xi',\bar{\xi},\bar{\xi}'\in L^1(\Omega, \mathcal{F},\mathbb{P};\dbR)$,
$t\in[0,T]$, $y,z,\widetilde{p}^0, \tilde{p}^1,\widetilde{q}^0,\widetilde{P}^0, \widetilde{P}^1,p^0, p^1,q^0,P^0,$ \\
$P^1,\gamma\in \mathbb{R}$, and $v\in V$, $i\in \mathcal{I}$,
\begin{align} \nonumber
\begin{aligned}
&H(t,\xi,\xi',y,z,v,\bar{\xi},\bar{\xi}',\bar{v},\widetilde{p}^0, \tilde{p}^1,\widetilde{q}^0,\widetilde{P}^0, \widetilde{P}^1,p^0, p^1,q^0,P^0, P^1,\gamma,i)\\
&=(\widetilde{p}^0+\gamma p^0) b(t,\xi,\xi',v,i)+(\widetilde{q}^0+\gamma q^0) \sigma(t,\xi,\xi',v,i)
        +(\widetilde{p}^1+\gamma p^1) \mathbb{E}[b(t,\xi,\xi',v,i)|\sF^\alpha_{t-}]\\
&\quad+\frac{1}{2}(\widetilde{P}^0+\gamma P^0)(\sigma(t,\xi,\xi',v,i)-\sigma(t,\bar{\xi},\bar{\xi}',\bar{v},i))^2,
\\
&\quad+\frac{1}{2} (\widetilde{P}^1+\gamma P^1)\mathbb{E}\Big[(\sigma(t,\xi,\xi',v,i)
-\sigma(t,\bar{\xi},\bar{\xi}',\bar{v},i))^2\Big|\sF^\alpha_{t-}\Big]\\
&\quad+\gamma f(t,\xi,\xi',y,z+p^0(\sigma(t,\xi,\xi',v,i)-\sigma(t,\bar{\xi},\bar{\xi}',\bar{v},i)),v,i).
\\
\end{aligned}
\end{align}
Thanks to (\ref{6.5-1113333}) and (\ref{9.15}), we derive
\begin{align} \nonumber
\begin{aligned}
&0\leq \mathbb{E}\bigg[\int_0^T\Big\{
H(t,X^\kappa(t),\widehat{X^\kappa}(t),Y^\kappa(t),Z^\kappa(t),v,X^\kappa(t),\widehat{X^\kappa}(t),v^\kappa(t),
\widetilde{p}^{0,\kappa}(t), \tilde{p}^{1,\kappa}(t),\widetilde{q}^{0,\kappa}(t),\\
&\qq\qq\qq
\widetilde{P}^{0,\kappa}(t), \widetilde{P}^{1,\kappa}(t),p^{0,\kappa}(t), p^{1,\kappa}(t),q^{0,\kappa}(t),P^{0,\kappa}(t), P^{1,\kappa}(t),\Upsilon^\kappa(t), \alpha(t-))\\
&\quad-H(t,X^\kappa(t),\widehat{X^\kappa}(t),Y^\kappa(t),Z^\kappa(t),v^\kappa(t),X^\kappa(t),\widehat{X^\kappa}(t),v^\kappa(t),
\widetilde{p}^{0,\kappa}(t), \tilde{p}^{1,\kappa}(t),\widetilde{q}^{0,\kappa}(t),\\
&\qq\qq\qq
\widetilde{P}^{0,\kappa}(t), \widetilde{P}^{1,\kappa}(t),p^{0,\kappa}(t), p^{1,\kappa}(t),q^{0,\kappa}(t),P^{0,\kappa}(t), P^{1,\kappa}(t),\Upsilon^\kappa(t)),\alpha(t-) \Big\}\mathbf{1}_{E_\varepsilon}(t)dt\bigg]\\
&\q+\sqrt{\kappa}\varepsilon+o(\varepsilon).
\end{aligned}
\end{align}
Recall $|E_\varepsilon|=\varepsilon$, one has from the arbitrariness of $\varepsilon$, for all $v\in V$, $\mathbb{P}$-a.s, a.e.
\begin{align} \nonumber
\begin{aligned}
&H(t,X^\kappa(t),\widehat{X^\kappa}(t),Y^\kappa(t),Z^\kappa(t),v,X^\kappa(t),\widehat{X^\kappa}(t),v^\kappa(t),\widetilde{p}^{0,\kappa}(t), \tilde{p}^{1,\kappa}(t),\widetilde{q}^{0,\kappa}(t),\\
&\qq\qq\qq
\widetilde{P}^{0,\kappa}(t), \widetilde{P}^{1,\kappa}(t),p^{0,\kappa}(t), p^{1,\kappa}(t),q^{0,\kappa}(t),P^{0,\kappa}(t), P^{1,\kappa}(t),\Upsilon^\kappa(t),\alpha(t-))\\
&\geq H(t,X^\kappa(t),\widehat{X^\kappa}(t),Y^\kappa(t),Z^\kappa(t),v^\kappa(t),X^\kappa(t),\widehat{X^\kappa}(t),v^\kappa(t),\widetilde{p}^{0,\kappa}(t), \tilde{p}^{1,\kappa}(t),\widetilde{q}^{0,\kappa}(t),\\
&\qq\qq\qq
\widetilde{P}^{0,\kappa}(t), \widetilde{P}^{1,\kappa}(t),p^{0,\kappa}(t), p^{1,\kappa}(t),q^{0,\kappa}(t),P^{0,\kappa}(t), P^{1,\kappa}(t),\Upsilon^\kappa(t),\alpha(t-)) -\sqrt{\kappa}.
\end{aligned}
\end{align}

From the definitions of  $\lambda_\kappa$ and $\mu_\kappa$ (see (\ref{9.6-2})), we know $|\lambda_\kappa|^2+|\mu_\kappa|^2=1.$
Consequently, there exists a subsequent of $(\lambda_\kappa,\mu_\kappa)$, still denoted by  $(\lambda_\kappa,\mu_\kappa)$,
 converging to  $(\lambda,\mu)$ with $|\lambda|^2+|\mu|^2=1$, as $\kappa\rightarrow\infty$.
According to the item $\mathrm{ii)}$ of (\ref{9.4}), one could choose a subsequent of  $(\lambda_\kappa,\mu_\kappa)$ such that
\begin{align} \nonumber
\begin{aligned}
&(X^\kappa(\cdot),\widehat{X^\kappa}(\cd),Y^\kappa(\cdot),Z^\kappa(\cdot), v^\kappa(\cdot), \widetilde{p}^{0,\kappa}(\cdot), \tilde{p}^{1,\kappa}(\cdot),\widetilde{q}^{0,\kappa}(\cdot),
 \widetilde{P}^{0,\kappa}(\cdot), \widetilde{P}^{1,\kappa}(\cdot),p^{0,\kappa}(\cdot), p^{1,\kappa}(\cdot),\\
 &\q\q\q\q        q^{0,\kappa}(\cdot),P^{0,\kappa}(\cdot), P^{1,\kappa}(\cdot),\Upsilon^\kappa(\cdot))\rightarrow\\
&(\bar{X} (\cdot),\widehat{X}(\cd),\bar{Y}(\cdot),\bar{Z} (\cdot), \bar{v}(\cdot), \widetilde{p}^{0}(\cdot), \tilde{p}^{1}(\cdot),\widetilde{q}^{0}(\cdot),
 \widetilde{P}^{0}(\cdot), \widetilde{P}^{1}(\cdot),p^{0}(\cdot), p^{1}(\cdot),q^{0}(\cdot),P^{0}(\cdot), P^{1}(\cdot),\Upsilon(\cdot)),
\end{aligned}
\end{align}
where $((\widetilde{p}^{0}, \tilde{p}^{1})^\intercal,(\widetilde{q}^{0},\widetilde{q}^{1})^\intercal)$ is the solution of the following BSDE:
\begin{align} \nonumber
\left\{
\begin{aligned}
d
\begin{pmatrix}
\widetilde{p}^{0}(t)\\
\widetilde{p}^{1}(t)
\end{pmatrix}
&=-\bigg[\begin{pmatrix}
b_x(t)& 0\\
b_{x'}(t)&
b_{x}(t)+ \widehat{b}_{x'}(t)
\end{pmatrix}
\begin{pmatrix}
\widetilde{p}^{0}(t)\\
\widetilde{p}^{1}(t)
\end{pmatrix}
+\begin{pmatrix}
\sigma_x(t)& 0\\
\sigma_{x'}(t)& 0
\end{pmatrix}\begin{pmatrix}
\widetilde{q}^{0}(t)\\
\widetilde{q}^{1}(t)
\end{pmatrix}\bigg]dt
+\begin{pmatrix}
\widetilde{q}^{0}(t)\\
\widetilde{q}^{1}(t)
\end{pmatrix}
dW(t),\\
\begin{pmatrix}
\widetilde{p}^{0}(T)\\
\widetilde{p}^{1}(T)
\end{pmatrix}
&=\begin{pmatrix}
\Phi_x(T)\\
\Phi_{x'}(T)
\end{pmatrix},
\end{aligned}
\right.
\end{align}
and $(\widetilde{P}, \widetilde{Q})=((\widetilde{P}^{0},\widetilde{P}^{1})^\intercal,\ (\widetilde{Q}^{0},\widetilde{Q}^{1})^\intercal)$ is the solution of BSDE:
\begin{align} \nonumber
\left\{
\begin{aligned}
d\widetilde{P}(t)&=-\bigg\{G^{\widetilde{P}}(t)\widetilde{P}(t)
+G^{\widetilde{Q}}(t)\widetilde{Q}(t)+G^{\tilde{p}}(t)\tilde{p}(t)+G^{\widetilde{q}}(t)\widetilde{q}(t)\bigg\}d+\widetilde{Q}(t)dW(t),\\
P(T)&=G^\Phi(T),
\end{aligned}
\right.
\end{align}
where
\begin{align} \nonumber
\begin{aligned}
\tilde{p}(t)=&(\widetilde{p}^{0}(t), \tilde{p}^{1}(t))^\intercal,\   \tilde{q}(t)=(\widetilde{q}^{0}(t),\widetilde{q}^{1}(t))^\intercal,\\
G^{\widetilde{P}}(t)=&
\begin{pmatrix}
2b_x(t)+(\sigma_x(t))^2& 0  \\
0& 2b_x(t)+(\sigma_x(t))^2&   \\
\end{pmatrix}, \q
G^{\widetilde{Q}}(t)=
\begin{pmatrix}
2\sigma_x(t)& 0  \\
0& 0
\end{pmatrix}, \\
G^{\widetilde{p}}(t)=&
\begin{pmatrix}
b_{xx}(t)& 0\\
0 &b_{xx}(t)
\end{pmatrix},\q
G^{\widetilde{q}}(t)=
\begin{pmatrix}
\sigma_{xx} & 0\\
0 &0
\end{pmatrix}, \q
G^\Phi(T)=
\begin{pmatrix}
\Phi_{xx}(T)\\
0
\end{pmatrix}.
\end{aligned}
\end{align}

\begin{theorem}
Let \textbf{Assumptions 1}-\textbf{4}  be in force. Let
$\bar{v}$ be the optimal control of the problem  (\ref{3.1})-(\ref{3.3})-(\ref{3.4})-(\ref{3.5})-(\ref{9.1}). By
$\bar{X}, (\bar{Y},\bar{Z})$ we denote the corresponding solutions to (\ref{3.1}) and (\ref{3.3}) with the optimal control $\bar{v}$,
respectively. Then there
exists two constants $\lambda, \mu$ satisfying $|\lambda|^2+|\mu|^2=1$ such that, for $v\in V$, $\mathbb{P}$-a.s., a.e.,

\begin{align} \nonumber
\begin{aligned}
&H(t,\bar{X}(t),\widehat{\bar{X}}(t),\bar{Y}(t),\bar{Z}(t),v,\bar{X}(t),\widehat{\bar{X}}(t),\bar{v}(t),\widetilde{p}^{0}(t), \tilde{p}^{1}(t),\widetilde{q}^{0}(t), \widetilde{P}^{0}(t), \widetilde{P}^{1}(t),\\
&\qq  p^{0}(t), p^{1}(t),q^{0}(t),P^{0}(t), P^{1}(t),\Upsilon(t)(t),\alpha(t-))\\
&\geq H(t,\bar{X}(t),\widehat{\bar{X}}(t),\bar{Y}(t),\bar{Z}(t),\bar{v}(t),\bar{X}(t),\widehat{\bar{X}}(t),\bar{v}(t),\widetilde{p}^{0}(t), \tilde{p}^{1}(t),\widetilde{q}^{0}(t),\widetilde{P}^{0}(t), \widetilde{P}^{1}(t),\\
&\qq  p^{0}(t), p^{1}(t),q^{0}(t),P^{0}(t), P^{1}(t),\Upsilon(t),\alpha(t-)).
\end{aligned}
\end{align}
\end{theorem}

\section{Concluding remarks}

\indent $\mathrm{i)}$ If $f$ is independent of $(y,z)$,
 the  relation (\ref{1.0}) can be replaced by the following ``weaker" relation
 \begin{equation}\label{11-1}
\mathbb{E}[Y^{1,\varepsilon}(t)]= \mathbb{E}[p_0(t)X^{1,\varepsilon}(t)+p_1(t)\mathbb{E}[X^{1,\varepsilon}(t)|\sF^\alpha_{t-}]].
\end{equation}
In addition, notice that (\ref{11-1}) can be written as
\begin{align} \nonumber
\mathbb{E}[Y^{1,\varepsilon}(t)]= \mathbb{E}[(p_0(t)+\mathbb{E}[p_1(t)|\sF^\alpha_{t-}])X^{1,\varepsilon}(t)].
\end{align}
According to the item $\mathrm{ii)}$ of \autoref{re 3.1}, our adjoint equation (\ref{5.2-2}) is just the equation (4.1) \cite{Nguyen-Yin-Nguyen-21}.
In the meantime,  our second-order adjoint equation (\ref{5.4}) naturally  reduces to  the  BSDE (4.2) \cite{Nguyen-Yin-Nguyen-21}, i.e.,
$P_1=Q_1=0$.

\indent $\mathrm{ii)}$  In \cite{Nguyen-Yin-Nguyen-21}, the cost functional is of the form
$$
J(t_0,\alpha_0,v(\cdot))=\mathbb{E}\bigg[
\int_0^Tf(s,X(s), \mathbb{E}[X(s)|\sF^\alpha_{s-}],v(s),\alpha(s-))ds
+\Phi(X(T), \mathbb{E}[X(T)|\sF^\alpha_{T-}],\alpha(T))\bigg].
$$
If we define
$$
Y(t):=\mathbb{E}\bigg[
\int_t^Tf(s,X(s), \mathbb{E}[X(s)|\sF^\alpha_{s-}],v(s),\alpha(s-))ds
+\Phi(X(T), \mathbb{E}[X(T)|\sF^\alpha_{T-}],\alpha(T))\Big|\sF_t\bigg],
$$
then
$$
\begin{aligned}
&Y(t)+\int_0^tf(s,X(s), \mathbb{E}[X(s)|\sF^\alpha_{s-}],v(s),\alpha(s-))ds\\
&=\mathbb{E}\bigg[
\int_0^Tf(s,X(s), \mathbb{E}[X(s)|\sF^\alpha_{s-}],v(s),\alpha(s-))ds
+\Phi(X(T), \mathbb{E}[X(T)|\sF^\alpha_{T-}],\alpha(T))\Big|\sF_t\bigg],
\end{aligned}
$$
is an $(\mathbb{F},\mathbb{P})$-martingale. From the martingale representation theorem
(Proposition 3.9 \cite{Donnelly-Heunis-12}), for each $t\in[0,T]$ there exists a unique
pair $(Z(\cdot),K(\cdot))\in \mathcal{H}_{\mathbb{F}}^{2}(0,T;\mathbb{R}^{n\times d})\times {\mathcal{K}}_{\mathbb{F}}^{2}(0,T;\mathbb{R}^n)$,
where ${\mathcal{K}}_{\mathbb{F}}^{2}(0,T;\mathbb{R}^n)$ is the family of $k=(k_{ij})_{i,j\in \mathcal{I}}$
such that the $\dbF$-progressively measurable process $k_{ij}$ satisfies  $k_{ii}=0$ and
$$\mathbb{E}\Big[ \int^T_0 \sum\limits_{i,j\in \mathcal{I}} |k_{ij}(t)|^2\lambda_{ij}\mathbf{1}_{\{\alpha(t-)=i\}}dt \Big]< \infty,$$
satisfying

$$
\begin{aligned}
Y(t)&=\Phi(X(T), \mathbb{E}[X(T)|\sF^\alpha_{T-}],\alpha(T))-
\int_t^Tf(s,X(s), \mathbb{E}[X(s)|\sF^\alpha_{s-}],v(s),\alpha(s-))ds\\
&\q +\int_t^T Z(s)dW(s)+\int_t^T\sum\limits_{i,j\in\mathcal{I}}K_{ij}(s)dM_{ij}(s), \ t\in[0,T].
\end{aligned}
$$
Since the above BSDE involves the martingale term of regime switching, the first- and second-order adjoint equations in
\cite{Nguyen-Yin-Nguyen-21} are also (conditional mean-field) BSDEs with regime switching.
In our case the BSDE (\ref{3.3}) is a classical BSDE (without the martingale term of regime switching), hence, the
 the first- and second-order adjoint equations are also classical BSDEs.

\section{Appendix}

\subsection{A basic estimate for BSDEs}

In this subsection let  $\gamma>1$. Consider the following BSDE:
\begin{equation}\label{6.1}
\begin{aligned}
Y(t)&=\xi+\int_t^TF(s,Y(s),Z(s),\alpha(s-))ds-\int_t^TZ(s)dW(s), \ t\in[0,T].
\end{aligned}
\end{equation}

Assume
$$
F: \Omega\times [0,T]\times \mathbb{R}^m\times\mathbb{R}^{m\times d}\times \mathcal{I}\rightarrow  \mathbb{R}^m
$$
satisfies
 \begin{description}\item[ ]\textbf{Assumption 5.}
 {\rm(i)}   There exists some constant $L>0$ such that for $t\in[0,T]$, $y,y'\in \mathbb{R}^m$, $z,z'\in \mathbb{R}^{m\times d}$, $i\in \mathcal{I}$,
$$
|F(t,y,z,i)-F(t,y',z',i)|\leq L(|y-y'|+|z-z'|).
$$
 {\rm(ii)}  $\mathbb{E}\Big[\Big(\int_0^T |F(t,0,0,i)|dt\Big)^\gamma\Big]<\infty.$
 \end{description}

\begin{lemma}\label{le 6.1}
Under \textbf{Assumption 5},  for $\xi\in L^\gamma(\Omega, \sF_T,\mathbb{P},\mathbb{R}^m)$, the BSDE (\ref{6.1})
exists a unique  solution $(Y(\cdot),Z(\cdot))\in \mathcal{S}_\mathbb{F}^\gamma(0,T;\mathbb{R}^m)\times
\mathcal{H}_{\mathbb{F}}^{2, \frac{\gamma}{2}}(0,T;\mathbb{R}^{m\times d})$.
\end{lemma}

\begin{lemma}\label{le 6.2}
Assume $F,\bar{F}$ satisfy \textbf{Assumptions 5} and $ \xi, \bar{\xi}\in L^\gamma(\Omega,\sF_T,\mathbb{P},\mathbb{R}^m)$.
Let $(Y(\cdot),Z(\cdot))$ and $(\bar{Y}(\cdot),\bar{Z}(\cdot))$
be the solutions to the equation (\ref{6.1}) with parameters  $(\xi, F)$ and $(\bar{\xi}, \bar{F})$, respectively.
Then there exists some positive constant $C$ depending on $\gamma,T, L$ such that
 $$
 \begin{aligned}
 &\mathbb{E}\bigg[\sup\limits_{t\in[0,T]}|Y(t)-\bar{Y}(t)|^\gamma+\Big(\int_0^T|Z(t)-\bar{Z}(t)|^2dt\Big)^\frac{\gamma}{2}\bigg]\\
 &\leq C\mathbb{E}\Big[ |\xi-\bar{\xi}|^\gamma+\Big(\int_0^T|F(t,\bar{Y}(t),\bar{Z}(t),\alpha(t-))
 -\bar{F}(t,\bar{Y}(t),\bar{Z}(t),\alpha(t-))| dt\Big)^\gamma\Big].
 \end{aligned}
 $$
In  particular, for $\bar{\xi}=\bar{F}\equiv0$,
\begin{equation}\label{6.2}
 \begin{aligned}
 &\mathbb{E}\bigg[\sup\limits_{t\in[0,T]}|Y(t)|^\gamma+\Big(\int_0^T|Z(t)|^2dt\Big)^\frac{\gamma}{2}\bigg]
 \leq C\mathbb{E}\Big[ |\xi|^\gamma+\Big(\int_0^T|F(t,0,0,\alpha(t-))| dt\Big)^\gamma\Big].
 \end{aligned}
\end{equation}
\end{lemma}

The proofs of the above two lemmata are similar to Proposition 3.2, Theorem 4.2 \cite{Briand-Delyon-Hu-Pardoux-Stoica-03}.
We omit them.

\subsection{Proof of  \autoref{pro 4.5}}

For $\mathrm{i)}$, from (\ref{6.2}), we have
$$
\begin{aligned}
&\mathbb{E}\bigg[\sup\limits_{t\in[0,T]}|Y^{1,\varepsilon}(t)|^\beta+\Big(\int_0^T|Z^{1,\varepsilon}(t)|^2dt\Big)^{\frac{\beta}{2}}\bigg]\\
&\leq C_\beta\mathbb{E}\bigg[ |\Phi_x(T)X^{1,\varepsilon}(T)+\Phi_{x'}(T)\widehat{X}^{1,\varepsilon}(T)|^\beta
+\bigg(\int_0^T\Big|f_x(t)X^{1,\varepsilon}(t)+f_{x'}(t)\widehat{X}^{1,\varepsilon}(t)\\
&\quad-\mathbf{1}_{E_\varepsilon}(t)\Big( f_z(t)p^0(t)\delta\sigma(t,v(t))
+q^0(t)\delta\sigma(t,v(t))+p^0(t)\delta b(t,v(t))+p^1(t)\widehat{\delta b(t,v(t))}\Big)\Big|dt\bigg)^\beta\bigg].
\end{aligned}
$$

\ms

Since $|\F_x(T)|+|\F_{x'}(T)|\leq(1+|\bar{X}(T)|+| \widehat{\bar{X}}(T)|)$, we get from H\"{o}lder inequality
$$
\begin{aligned}
&\mathbb{E}\bigg[ |\Phi_x(T)X^{1,\varepsilon}(T)+\Phi_{x'}(T)\widehat{X}^{1,\varepsilon}(T)|^\beta\bigg]\\
&\leq C\dbE\[(1+|\bar{X}(T)|^\b+| \widehat{\bar{X}}(T)|^\b)\cd (|X^{1,\varepsilon}(T)|+|\widehat{X}^{1,\varepsilon}(T)|)^\b  \]\\
&\leq C\bigg\{ \dbE\[ 1+|\bar{X}(T)|^{2\b}+| \widehat{\bar{X}}(T)|^{2\b}  \]  \bigg\}^\frac{1}{2}\cd
\bigg\{\dbE\[\sup_{t\in[0,T]}|X^{1,\e}|^{2\b}+\sup_{t\in[0,T]}|\widehat{X}^{1,\e}|^{2\b}\]  \bigg\}^\frac{1}{2}\leq C\e^{\frac{\b}{2}}.
\end{aligned}
$$

Let us estimate the terms
 $$\mathbb{E}\bigg[\Big(\int_{E_\varepsilon}|q^0(t)\delta\sigma(t,v(t))|dt\Big)^\beta\bigg],\q
 \mathbb{E}\bigg[\Big(\int_{E_\varepsilon}|p^1(t)\widehat{\delta b(t,v(t))}|dt\Big)^\beta\bigg].$$
The other terms can be calculated similarly.

First, as for the term $\mathbb{E}\bigg[\Big(\int_{E_\varepsilon}|q^0(t)\delta\sigma(t,v(t))|dt\Big)^\beta\bigg]$,
since $|\delta\sigma(t,v(t))|\leq L(1+|\bar{X}(t)|+|\widehat{\bar{X}}(t)|+|v(t)|+|\bar{v}(t)|)$, we have from   H\"{o}lder inequality
\begin{align} \nonumber
\begin{aligned}
&\mathbb{E}\bigg[\Big(\int_{E_\varepsilon}|q^0(t)\delta\sigma(t,v(t))|dt\Big)^\beta\bigg]\\
&\leq C_\beta \mathbb{E}\bigg[\bigg(\int_{E_\varepsilon}|q^0(t)|\big(1+|\bar{X}(t)|+|\widehat{\bar{X}}(t)|+|v(t)|
+|\bar{v}(t)|\big)dt\bigg)^\beta\bigg]\\
&\leq C_\beta\mathbb{E}\bigg[\bigg(\int_{E_\varepsilon}|q^0(t)|^2dt\bigg)^{\frac{\beta}{2}}
\bigg(\int_{E_\varepsilon}\big(1+|\bar{X}(t)|^2+|\hat{\bar{X}}(t)|^2+|v(t)|^2+|\bar{v}(t)|^2\big)dt\bigg)^\frac{\beta}{2}\bigg]\\
&\leq C_\beta\bigg\{\mathbb{E}\bigg[\bigg(\int_{E_\varepsilon}|q^0(t)|^2dt\bigg)^{\frac{4\beta}{8-\beta}}\bigg] \bigg\}^\frac{8-\beta}{8}
\bigg\{\mathbb{E}\bigg[\bigg(\int_{E_\varepsilon}\big(1+|\bar{X}(t)|^2+|\hat{\bar{X}}(t)|^2+|v(t)|^2
+|\bar{v}(t)|^2\big)dt\bigg)^4\bigg]\bigg\}^\frac{\beta}{8}\\
&\leq C_\beta\bigg\{\mathbb{E}\bigg[\bigg(\int_{E_\varepsilon}|q^0(t)|^2dt\bigg)^{\frac{4\beta}{8-\beta}}\bigg] \bigg\}^\frac{8-\beta}{8}
\bigg\{\varepsilon^3\mathbb{E}\bigg[ \int_{E_\varepsilon}\big(1+|\bar{X}(t)|^8+|\hat{\bar{X}}(t)|^8+|v(t)|^8
+|\bar{v}(t)|^8\big)dt \bigg]\bigg\}^\frac{\beta}{8}\\
&\leq \rho(\varepsilon)\varepsilon^{\frac{\beta}{2}},
\end{aligned}
\end{align}
where $\rho(\e)=C_\beta\bigg\{\mathbb{E}\bigg[\bigg(\int_{E_\varepsilon}|q^0(t)|^2dt\bigg)
^{\frac{4\beta}{8-\beta}}\bigg] \bigg\}^\frac{8-\beta}{8}\bigg\{\sup_{t\in[0,T]}\mathbb{E}\bigg[1+|\bar{X}(t)|^8+|\hat{\bar{X}}(t)|^8+|v(t)|^8
+|\bar{v}(t)|^8\bigg]\bigg\}^\frac{\beta}{8}\rightarrow0$ as $\e\rightarrow0$.
\ms

Now we focus on the term   $\mathbb{E}\bigg[\Big(\int_{E_\varepsilon}|p^1(t)\widehat{\delta b(t,v(t))}|dt\Big)^\beta\bigg]$.
Notice that $|\delta b(t,v(t))|\leq L(1+|\bar{X}(t)|+|\widehat{\bar{X}}(t)|+|v(t)|+|\bar{v}(t)|)$,
it follows from H\"{o}lder's inequality
\begin{align} \nonumber
\begin{aligned}
&\mathbb{E}\bigg[\Big(\int_{E_\varepsilon}|p^1(t)\widehat{\delta b(t,v(t))}|dt\Big)^\beta\bigg]
\leq \mathbb{E}\bigg[\sup\limits_{0\leq t\leq T}|p_1(t)|^\beta\cdot
\Big(\int_{E_\varepsilon} \mathbb{E}\Big[\delta b(t,v(t))|\sF^\alpha_{t-}\Big] dt  \Big)^\beta\bigg] \\
&\leq \varepsilon^{\frac{\beta}{2}}
\bigg\{  \mathbb{E}\bigg[\Big(\int_{E_\varepsilon}\Big |\mathbb{E}\Big[\delta b(t,v(t))|\sF^\alpha_{t-}\Big]\Big|^2 dt\Big)^4\bigg]\bigg\}^{\frac{\beta}{8}}\cdot
\bigg\{ \mathbb{E}\Big[ \sup\limits_{0\leq t\leq T}|p_1(t)|^{\frac{8\beta}{8-\beta}} \Big]\bigg\}^\frac{8-\beta}{8}\\
&\leq C_\beta\varepsilon^{\frac{7\beta}{8}}
\bigg\{ \int_{E_\varepsilon}(1+\mathbb{E}[|\bar{X}(t)|^8]+\mathbb{E}[|\widehat{\bar{X}}(t)|^8]+\mathbb{E}[|v(t)|^8]+\mathbb{E}[|\bar{v}(t)|^8])dt\bigg\}^ {\frac{\beta}{8}}\\
&\leq C_\beta\varepsilon^{\beta}.
\end{aligned}
\end{align}

Next, for  $\mathrm{ii)}$, it follows from the definition of $\delta^2 Y(s)$ and (\ref{4.0-1})
\begin{align} \nonumber
\begin{aligned}
\delta^2Y(s)&=\Phi_x^\rho(T) \delta^2X(T)+\Phi_{x'}^\rho(T) \delta^2\widehat{X}(T)+I_1(T)+\int_t^T (f^{\rho\varepsilon}_x(s)\delta^2X(s)+f^{\rho\varepsilon}_{x'}(s)\delta^2\widehat{X}(s)\\
&\q+f^{\rho\varepsilon}_{y}(s)\delta^2Y(s)+f^{\rho\varepsilon}_{z}(s)\delta^2Z(s) +I_2(s))ds-\int_t^T\delta^2Z(s)dW(s),\  s\in[t,T],
\end{aligned}
\end{align}
where
$$
\begin{aligned}
I_1(T)&=X^{1,\varepsilon}(T)\big(\Phi_x^\rho(T)-\Phi_x(T)\big)+\widehat{X}^{1,\varepsilon}(T)\big(\Phi_{x'}^\rho(T)-\Phi_{x'}(T)\big),\\
I_2(s)&=X^{1,\varepsilon}(s)\big(f_x^{\rho\varepsilon}(s)-f_x(s)\big)+\widehat{X}^{1,\varepsilon}(s)\big(f_{x'}^{\rho\varepsilon}(s)-f_{x'}(s)\big)
        +Y^{1,\varepsilon}(s)\big(f_{y}^{\rho\varepsilon}(s)-f_y(s)\big)\\
        &\q+Z^{1,\varepsilon}(s)\big(f_{z}^{\rho\varepsilon}(s)-f_z(s)\big)
        +\mathbf{1}_{E_\varepsilon}\Big\{\delta f(s,v(s))+f_z(s)p^0(s)\delta\sigma(s,v(s))\\
        &\q+q^0(s)\delta\sigma(s,v(s))+p^0(s)\delta b(s,v(s))+p^1(s)\widehat{\delta b(s,v(s))} \Big\}.
        \end{aligned}
$$
Thanks to   \autoref{le 6.2}, it yields
\begin{equation}\label{8.6}
\begin{aligned}
&\mathbb{E}\bigg[\sup\limits_{0\leq s\leq T }|\delta^2 Y(s)|^4+\Big(\int_0^T|\delta^2Z(s)|^2ds\Big)^2 \bigg]\\
&\leq \mathbb{E}|\Phi_{x}^\rho(T)\delta^2X(T)+\Phi_{x'}^\rho(T) \delta^2\widehat{X}(T)+I_1(T)|^4\\
&+\mathbb{E}\Big[ \Big(\int_0^T|f^{\rho\varepsilon}_x(s)\delta^2X(s)+f^{\rho\varepsilon}_{x'}(s)\delta^2\widehat{X}(s)+I_2(s)|ds\Big)^4\Big].
\end{aligned}
\end{equation}

\ms

We estimate it term by term.

\ms

$\mathrm{a)}$\  Thanks to   $\Phi_x,\Phi_{x'}$ being uniformly Lipschitz continuous in $(x,x')$, H\"{o}lder inequality  and the fact
$|\Phi_{x}^\rho(T)|+|\Phi_{x'}^\rho(T)|\leq L(1+|\bar{X}(T)|+|X^\e(T)|+|\widehat{\bar{X}}(T)|+|\widehat{X}^\e(T)|),$
we obtain from \autoref{le 4.1} and \autoref{pro 4.2}
\begin{equation}\label{8.6-1}
\begin{aligned}
&\mathbb{E}|\Phi_{x}^\rho(T)\delta^2X(T)+\Phi_{x'}^\rho(T) \delta^2\widehat{X}(T)+I_1(T)|^4\\
&\leq C \bigg\{\mathbb{E}|\delta^2X(T)|^8\bigg\}^{\frac{1}{2}}
\bigg\{ 1+\mathbb{E}|\bar{X}(T)|^8+\mathbb{E}|X^\varepsilon(T)|^8+ \mathbb{E}|\widehat{\bar{X}}(T)|^8+ \mathbb{E}| \widehat{X}^\varepsilon(T)|^8          \bigg\}^{\frac{1}{2}}\\
&\q +C\bigg\{\mathbb{E}|X^{1,\varepsilon}(T)|^8\bigg\}^{\frac{1}{2}}
\bigg\{\mathbb{E}| \delta^1 X(T)|^8+\mathbb{E} | \delta^1 \widehat{X}(T)|^8         \bigg\}^{\frac{1}{2}}\leq C\varepsilon^4.
\end{aligned}
\end{equation}

$\mathrm{b)}$ Notice
$|f^{\rho\varepsilon}_x(s)|+|f^{\rho\varepsilon}_{x'}(s)|\leq L(1+|\bar{X}(s)|+|X^\varepsilon(s)|+|\widehat{\bar{X}}(s)|
+| \widehat{X}^\varepsilon(s)|+|v^\e(s)|)$, thanks to the item $\mathrm{ii)}$ of (\ref{4.3}) and H\"{o}lder  inequality, one gets
\begin{equation}\label{8.8}
\begin{aligned}
&\mathbb{E}\Big[ \Big(\int_0^T|f^{\rho\varepsilon}_x(s)\delta^2X(s)+f^{\rho\varepsilon}_{x'}(s)\delta^2\widehat{X}(s)|ds\Big)^4\Big]\\
&\leq C\bigg\{\dbE\[\sup_{s\in[0,T]} |\delta^2X(s)|^8+\sup_{s\in[0,T]} |\delta^2\widehat{X}(s)|^8\]\bigg\}^\frac{1}{2}\cd
\bigg\{\dbE\[\sup_{s\in[0,T]}(1+|\bar{X}(s)|^8+|X^\varepsilon(s)|^8\\
&\qq\qq\qq\qq\qq\qq\qq\qq\qq\qq\qq\qq
+|\widehat{\bar{X}}(s)|^8+|\widehat{X}^\varepsilon(s)|^8+|v^\e(s)|^8)\]\bigg\}^\frac{1}{2}\\
&\leq C\varepsilon^4.
\end{aligned}
\end{equation}

\ms

We now analyse $\mathbb{E}\Big[ \Big(\int_0^TI_2(s)|ds\Big)^4\Big]$. First, let us estimate
those terms in $I_2$ without involving $\mathbf{1}_{E_\varepsilon}$.

\ms

$\mathrm{c)}$\ Notice
$|f_{z}^{\rho\varepsilon}(s)-f_z(s)|\leq L(|\delta^1 X(s)|+|\delta^1 \widehat{X}(s)|+|\delta^1 Y(s)|+|\delta^1 Z(s)|+
|\delta f_z(s,v(s))|\mathbf{1}_{E_\varepsilon}(s))$
and
$|\delta f_z(s,v(s))|\leq L(1+|\bar{X}(s)|+|\widehat{\bar{X}}(s)|
+| v(s)|+|\bar{v}(s)|)$,
we have from  H\"{o}lder inequality, \autoref{le 4.1} as well as the item $\mathrm{i)}$ of \autoref{pro 4.5},
\begin{equation}\label{8.7}
\begin{aligned}
&\mathbb{E}\Big[ \Big(\int_0^T|Z^{1,\varepsilon}(s)\big(f_{z}^{\rho\varepsilon}(s)-f_z(s)\big) |ds\Big)^4\Big]\\
&\leq  \mathbb{E}\Big[ \Big(\int_0^T|Z^{1,\varepsilon}(s)|^2 ds \Big)^2 \Big(\int_0^T|f_{z}^{\rho\varepsilon}(s)-f_z(s)|^2 ds \Big)^2 \Big]\\
&\leq \Big\{\mathbb{E}\Big[ \Big(\int_0^T|Z^{1,\varepsilon}(s)|^2 ds \Big)^4 \Big]  \Big\}^\frac{1}{2}
\cdot\Big\{\mathbb{E}\Big[ \Big(\int_0^T|f_{z}^{\rho\varepsilon}(s)-f_z(s)|^2 ds \Big)^4 \Big]  \Big\}^\frac{1}{2}\\
&\leq C\bigg\{\mathbb{E}\Big[ \Big(\int_0^T|Z^{1,\varepsilon}(s)|^2 ds \Big)^4 \Big]  \bigg\}^\frac{1}{2}\cdot
\bigg\{\mathbb{E}\bigg[\sup\limits_{0\leq s\leq T}\Big(|\delta^1 X(s)|^8+|\delta^1 \widehat{X}(s)|^8
  +|\delta^1 Y(s)|^8 \Big) \\
&\qq\qq\qq\qq\qq\qq\qq\qq\qq\qq
+\Big(\int_0^T|\delta^1 Z(s)|^2+|\delta f_z(s,v(s))|^2\mathbf{1}_{E_\varepsilon}(s)ds  \Big)^4 \bigg]  \bigg\}^\frac{1}{2}\\
\end{aligned}
\end{equation}
\begin{align}\nonumber
\begin{aligned}
 &\leq C\bigg\{\mathbb{E}\Big[ \Big(\int_0^T|Z^{1,\varepsilon}(s)|^2 ds \Big)^4 \Big]  \bigg\}^\frac{1}{2} \cd
 \bigg\{\mathbb{E}\bigg[\sup\limits_{0\leq s\leq T}\Big(|\delta^1 X(s)|^8+|\delta^1 \widehat{X}(s)|^8+|\delta^1 Y(s)|^8 \Big)\\
&\qq\qq\qq\qq\qq
+\Big(\int_0^T|\delta^1 Z(s)|^2+(1+|\bar{X}(s)|+|\widehat{\bar{X}}(s)|
+| v(s)|+|\bar{v}(s)|)^2\mathbf{1}_{E_\varepsilon}(s)ds  \Big)^4 \bigg]  \bigg\}^\frac{1}{2} \\
  &\leq C\varepsilon^4.
\end{aligned}
\end{align}
Similar to (\ref{8.7}), it yields
\begin{equation}\label{8.7-11}
\begin{aligned}
&\mathbb{E}\Big[ (\int_0^T|X^{1,\varepsilon}(s)\big(f_x^{\rho\varepsilon}(s)-f_x(s)\big) |ds)^4\Big]+
\mathbb{E}\Big[ (\int_0^T|\widehat{X}^{1,\varepsilon}(s)\big(f_{x'}^{\rho\varepsilon}(s)-f_{x'}(s)\big) |ds)^4\Big],\\
&+\mathbb{E}\Big[ (\int_0^T|Y^{1,\varepsilon}(s)\big(f_{y}^{\rho\varepsilon}(s)-f_y(s)\big) |ds)^4\Big]\leq C\varepsilon^4.
\end{aligned}
\end{equation}

\ms

Next, we  estimate those terms  in $I_2$ involving $\mathbf{1}_{E_\varepsilon}$.

\ms

$\mathrm{d)}$\ Since $|\delta\sigma(s,v(s))|\leq L(1+|\bar{X}(s)|+|\hat{\bar{X}}(s)|+|v(s)|+|\bar{v}(s)|)$,
it follows from H\"{o}lder inequality
\begin{align} \nonumber
\begin{aligned}
&\mathbb{E}\bigg[\Big(\int_{E_\varepsilon}|q^0(s)\delta\sigma(s,v(s))|ds\Big)^4\bigg]\\
&\leq C_\beta \mathbb{E}\bigg[\bigg(\int_{E_\varepsilon}|q^0(s)|\big(1+|\bar{X}(s)|+|\hat{\bar{X}}(s)|+|v(s)|
+|\bar{v}(s)|\big)ds\bigg)^4\bigg]\\
&\leq C_\beta\mathbb{E}\bigg[\bigg(\int_{E_\varepsilon}|q^0(s)|^2ds\bigg)^{2}
\bigg(\int_{E_\varepsilon}\big(1+|\bar{X}(s)|^2+|\hat{\bar{X}}(s)|^2+|v(s)|^2+|\bar{v}(s)|^2\big)ds\bigg)^2\bigg]\\
&\leq C_\beta\bigg\{\mathbb{E}\bigg[\bigg(\int_{E_\varepsilon}|q^0(s)|^2ds\bigg)^{4}\bigg] \bigg\}^\frac{1}{2}
\bigg\{\mathbb{E}\bigg[\bigg(\int_{E_\varepsilon}\big(1+|\bar{X}(s)|^2+|\hat{\bar{X}}(s)|^2+|v(s)|^2
+|\bar{v}(s)|^2\big)ds\bigg)^4\bigg]\bigg\}^\frac{1}{2}\\
&\leq C_\beta \varepsilon^{2}\bigg\{\mathbb{E}\bigg[\bigg(\int_{E_\varepsilon}|q^0(s)|^2ds\bigg)^{4} \bigg\}^\frac{1}{2}.
\end{aligned}
\end{align}

\noindent Dominated convergence theorem allows to show that $\bigg\{\mathbb{E}\bigg[\bigg(\int_{E_\varepsilon}|q^0(s)|^2ds\bigg)^{4} \bigg\}^\frac{1}{2}$
converges to zero, as $\varepsilon\rightarrow0$. Define $\rho(\varepsilon):= C_\beta\bigg\{\mathbb{E}\bigg[\bigg(\int_{E_\varepsilon}|q^0(s)|^2ds\bigg)^{4}\bigg] \bigg\}^\frac{1}{2}$, we have
$\mathbb{E}\bigg[\Big(\int_{E_\varepsilon}|q^0(s)\delta\sigma(s,v(s))|ds\Big)^4\bigg]\leq\varepsilon^2\rho(\varepsilon).$\\
Similarly, we can deduce
\begin{equation}\label{8.9}
\begin{aligned}
&\mathbb{E}\bigg[\Big(\int_{E_\varepsilon} \delta f(s,v(s))+f_z(s)p^0(s)\delta\sigma(s,v(s))+p^0(s)\delta b(s,v(s))+p^1(s)\widehat{\delta b(s,v(s))}ds\Big)^4\bigg]\leq \varepsilon^{2}\rho(\varepsilon).
\end{aligned}
\end{equation}
Hence, it yields
\begin{equation}\label{8.9-1}
\begin{aligned}
&\mathbb{E}\Big[ \Big(\int_0^T|I_2(s)|ds\Big)^4\Big]\leq \varepsilon^{2}\rho(\varepsilon).
\end{aligned}
\end{equation}
Finally, it follows from (\ref{8.6}), (\ref{8.6-1}), (\ref{8.8}) and (\ref{8.9-1}) that
\begin{equation}\label{8.611111111}
\begin{aligned}
&\mathbb{E}\bigg[\sup\limits_{0\leq s\leq T }|\delta^2 Y(s)|^4+\Big(\int_0^T|\delta^2Z(s)|^2ds\Big)^2 \bigg] \leq \varepsilon^{2}\rho(\varepsilon).
\end{aligned}
\end{equation}

\subsection{Proof of Lemma \ref{le 5.1}}

Denote
$$
\begin{aligned}
\bar{\Gamma}(t)&=(\bar{X}(t),\widehat{\bar{X}}(t),\bar{Y}(t),\bar{Z}(t)), \q \Gamma^\varepsilon(t)=(X^\varepsilon(t),\widehat{X}^\varepsilon(t),Y^\varepsilon(t),Z^\varepsilon(t)),
\end{aligned}
$$
and define
$$
\begin{aligned}
D^2f^{\tilde{\rho}\rho\varepsilon}(t)&=2\int_0^1\int_0^1\rho
D^2f(t,\bar{\Gamma}(t)+\tilde{\rho}\rho(\Gamma^\varepsilon(t)-\bar{\Gamma}(t)),v^\varepsilon(t),\alpha(t-))d\tilde{\rho}d\rho,\\
D^2\Phi^{\tilde{\rho}\rho\e}(t)&=2\int_0^1\int_0^1\rho
D^2\Phi(\bar{X}(T)+\tilde{\rho}\rho(X^\varepsilon(T)-\bar{X}(T)),\widehat{\bar{X}}(T)+\tilde{\rho}\rho(\widehat{X}^\varepsilon(T)
-\widehat{\bar{X}}(T)),   \alpha(T))d\tilde{\rho}d\rho.
\end{aligned}
$$
Then we obtain
\begin{align} \nonumber
\left\{
\begin{aligned}
d\delta^3Y(t)&=\Big(f_x(t)\delta^3X(t)+f_{x'}(t)\delta^3\widehat{X}(t)+f_{y}(t)\delta^3Y(t)+f_{z}(t)\delta^3Z(t)+M_1(t)\Big)dt
-\delta^3Z(t)dW(t),\\
\delta^3Y(T)&=\Phi_x(T)\delta^3X(T)+\Phi_{x'}(T)\delta^3\widehat{X}(T)+M_2(T),
\end{aligned}
\right.
\end{align}
where
\begin{align} \nonumber
\begin{aligned}
M_1(t)&=\Big[\delta f_x(t,v(t),p^0\delta\sigma(t))\delta^1X(t)+\delta f_{x'}(t,v(t),p^0\delta\sigma(t))\delta^1\widehat{X}(t)+\delta f_{y}(t,v(t),p^0\delta\sigma(t))\delta^1Y(t)\\
       &\quad +\delta f_{z}(t,v(t),p^0\delta\sigma(t))(\delta^1Z(t)-p^0(t)\delta \sigma(t,v(t))\mathbf{1}_{E_\varepsilon}(t))\Big]\mathbf{1}_{E_\varepsilon}(t)\\
     &\quad  +\frac{1}{2} [\delta^1X(t), \delta^1\widehat{X}(t),\delta^1Y(t),\delta^1Z(t)-p^0(t)\delta\sigma(t,v(t))\mathbf{1}_{E_\varepsilon}(t)]
  D^2f^{\tilde{\rho}\rho\varepsilon}(t)\\
  &\qq\qq\q
  [\delta^1X(t), \delta^1\widehat{X}(t),\delta^1Y(t),\delta^1Z(t)-p^0(t)\delta\sigma(t,v(t))\mathbf{1}_{E_\varepsilon}(t)]^\intercal\\
  &\quad-\frac{1}{2}[1, p^0(t),p^0(t)\sigma_x(t)+q^0(t)]
                         D_{xyz}^2f(t) [1, p^0(t),p^0(t)\sigma_x(t)+q^0(t)]^\intercal(X^{1,\varepsilon}(t))^2\\
M_2(T)&=\frac{1}{2}\Phi_{xx}^{\tilde{\rho}\rho\e}(T)(\delta^1 X(T))^2-\frac{1}{2}\Phi_{xx}(T)(X^{1,\varepsilon}(T))^2\\
     &\quad+\frac{1}{2}\Phi_{xx'}^{\tilde{\rho}\rho\e}(T)\delta^1 X(T)\delta^1 \widehat{X}(T)+\frac{1}{2}\Phi_{x'x}^{\tilde{\rho}\rho\e}(T)\delta^1 X(T)\delta^1 \widehat{X}(T)
    +\frac{1}{2}\Phi_{x'x'}^{\tilde{\rho}\rho\e}(T)(\delta^1 \widehat{X}(T))^2 ,
 \end{aligned}
\end{align}
and for $l=x,x',y,z$
 \begin{align} \nonumber
 \begin{aligned}
               &\d f_l(t,v(t),p^0\d\si(t))=\partial_lf\(t,\bar{X}(t),\dbE[\bar{X}(t)|\sF^\a_{t-}],\bar{Y}(t),\bar{Z}(t)+p^0(t)\d\si(t,v(t)),v(t),\a(t-)\)\\
               &\qq\qq\qq\qq\q -\partial_lf\(t,\bar{X}(t),\dbE[\bar{X}(t)|\sF^\a_{t-}],\bar{Y}(t),\bar{Z}(t),\bar{v}(t),\a(t-)\).
            \end{aligned}
            \end{align}
From the boundness of $\Phi_{xx},\Phi_{xx'},\Phi_{x'x'}$, H\"{o}lder inequality, \autoref{le 4.1} and (\ref{4.3}), we have
$$
\begin{aligned}
&\mathbb{E}\Big[|M_2(T)|^2\Big]\\
&\leq C\dbE\[|\Phi_{xx}^{\tilde{\rho}\rho\e}(T)-\Phi_{xx}(T)|^2 |\d^1 X(T)|^4+|\d^2 X(T)|^2|\d^1X(T)+X^{1,\e}(T)|^2
+|\d^1 X(T)\d^1\h{X}(T)|^2+|\d^1 \h{X}(T)|^4\]\\
&\leq C\bigg( \Big\{\dbE\[|\Phi_{xx}^{\tilde{\rho}\rho\e}(T)-\Phi_{xx}(T)|^4\]\Big\}^\frac{1}{2}
\Big\{\dbE\[\sup_{t\in[0,T]}|\d^1 X(t)|^8\]\Big\}^\frac{1}{2}\\
&\q+\Big\{\dbE\[\sup_{t\in[0,T]}|\d^2 X(t)|^4\]\Big\}^\frac{1}{2}
\Big\{\dbE\[\sup_{t\in[0,T]}|\d^1X(t)+X^{1,\e}(t)|^4\]\Big\}^\frac{1}{2}\\
&\q+\Big\{\mathbb{E}\[| \sup_{t\in[0,T]}|\delta^1 X(t)|^4\]\Big\}^\frac{1}{2}\cdot
    \Big\{\mathbb{E}\[|\sup_{t\in[0,T]}|\delta^1\widehat{ X}(t)|^4\]\Big\}^\frac{1}{2}
      +\mathbb{E}\[|\sup_{t\in[0,T]}|\delta^1\widehat{ X}(t)|^4\]\bigg).
      \end{aligned}
$$
\noindent Dominated convergence theorem can show  $\Big\{\dbE\[|\Phi_{xx}^{\tilde{\rho}\rho\e}(T)-\Phi_{xx}(T)|^4\]\Big\}^\frac{1}{2}\ra0$
as $\e\ra0$.
Thereby, it yields from \autoref{pro 4.2} that $\mathbb{E}\Big[|M_2(T)|^2\Big]\leq\e^2\rho(\e).$

\ms

Now we focus on $M_1(t)$.  Let us first analyse those terms in $M_1(t)$ involving  $\mathbf{1}_{E_\varepsilon}(t)$.

\ms

$\mathrm{a)}$\ Notice
$|\delta f_{z}(t,v(t),p^0\delta\sigma(t))|\leq L (1+|\bar{X}(t)|+|\widehat{\bar{X}}(t)|+|v(t)|+|\bar{v}(t)|)$,
we obtain from (\ref{4.1-2}),
$$
\begin{aligned}
&\mathbb{E}\bigg[\Big(\int_0^T\delta f_{z}(t,v(t),p^0\delta\sigma(t))(\delta^1Z(t)-p^0(t)\delta \sigma(t,v(t))\mathbf{1}_{E_\varepsilon}(t)) \mathbf{1}_{E_\varepsilon}(t)\Big)^2\bigg]\\
&= \mathbb{E}\bigg[\Big(\int_{{E_\varepsilon}}\delta f_{z}(t,v(t),p^0\delta\sigma(t))(\delta^2Z(t)
+Z^{1,\varepsilon}(t)-p^0(t)\delta \sigma(t,v(t))\mathbf{1}_{E_\varepsilon}(t))dt\Big)^2\bigg]\\
&= \mathbb{E}\bigg[\Big(\int_{{E_\varepsilon}}\delta f_{z}(t,v(t),p^0\delta\sigma(t))(\delta^2Z(t)
+[p^0(t)\sigma_x(t)+q^0(t)]X^{1,\varepsilon}(t)
           +[p^0(t)\sigma_{x'}(t)+q^1(t)]\widehat{X}^{1,\varepsilon}(t)
)dt\Big)^2\bigg]\\
&\leq L\mathbb{E}\bigg[\Big(\int_{{E_\varepsilon}}(1+|\bar{X}(t)|+|\widehat{\bar{X}}(t)|+|v(t)|+|\bar{v}(t)|)
(|\delta^2Z(t)
+[p^0(t)\sigma_x(t)+q^0(t)]X^{1,\varepsilon}(t)\\
           &\quad +[p^0(t)\sigma_{x'}(t)+q^1(t)]\widehat{X}^{1,\varepsilon}(t)|
)dt\Big)^2\bigg].
\end{aligned}
$$
From H\"{o}lder inequality,\ (\ref{4.4}) and \autoref{pro 4.2}, we get from Dominated convergence theorem
$$
\begin{aligned}
&\mathbb{E}\Big[ \Big(\int_{{E_\varepsilon}}|v(t)| |\delta^2Z(t)| dt  \Big)^2   \Big]
\leq \Big\{\varepsilon \mathbb{E}\Big[\int_{{E_\varepsilon}}|v(t)|^4dt\Big]\Big\}^\frac{1}{2}
\Big\{ \mathbb{E} \Big(\int_{{E_\varepsilon}}|\delta^2Z(t)|^2 dt    \Big)^2 \Big\}^\frac{1}{2}
\leq \varepsilon^3,
\end{aligned}
$$
and
$$
\begin{aligned}
&\mathbb{E}\Big[ \Big(\int_{{E_\varepsilon}}|v(t)| |q^0(t) X^{1,\varepsilon}| dt  \Big)^2   \Big]\\
&\leq \bigg\{\e^3\mathbb{E}\Big[\int_{{E_\varepsilon}}|v(t)|^8dt\Big]\bigg\}^\frac{1}{4}
\bigg\{ \mathbb{E} \Big(\int_{{E_\varepsilon}}|q^0(t)|^2 dt    \Big)^2 \bigg\}^\frac{1}{2}
\bigg\{ \mathbb{E} \Big[\sup_{0\leq t\leq T}|X^{1,\varepsilon}(t)|^8 dt  \Big] \bigg\}^\frac{1}{4}
\leq \varepsilon^2\rho(\varepsilon).
\end{aligned}
$$
The other terms can be estimated with the  similar argument. Consequently, one has
$$
\mathbb{E}\bigg[\Big(\int_0^T\delta f_{z}(t,v(t),p^0\delta\sigma(t))(\delta^1Z(t)-p^0(t)\delta \sigma(t,v(t))\mathbf{1}_{E_\varepsilon}(t))\mathbf{1}_{E_\varepsilon}(t)dt\Big)^2\bigg]\leq\varepsilon^2\rho(\varepsilon).
$$
Analogously, we can show
\begin{align} \nonumber
\begin{aligned}
&\mathbb{E}\bigg[\Big(\int_0^T[\delta f_x(t,v(t),p^0\delta\sigma(t))\delta^1X(t)+\delta f_{x'}(t,v(t),p^0\delta\sigma(t))\delta^1\widehat{X}(t)+\delta f_{y}(t,v(t),p^0\delta\sigma(t))\delta^1Y(t)]\mathbf{1}_{E_\varepsilon}(t)dt\Big)^2\bigg]\\
&\leq\varepsilon^2\rho(\varepsilon).
\end{aligned}
\end{align}

$\mathrm{b)}$\  Let us now concern those terms in $M_1(t)$  involving  the second-order derivatives of $f$. First,
$$
\begin{aligned}
&\mathbb{E}\bigg[\Big(\int_0^T
\frac{1}{2} [\delta^1X(t), \delta^1\widehat{X}(t),\delta^1Y(t),\delta^1Z(t)-p^0(t)\delta\sigma(t,v(t))\mathbf{1}_{E_\varepsilon}(t)]
  D^2f^{\tilde{\rho}\rho\varepsilon}(t)\\
  &\quad\quad
  [\delta^1X(t), \delta^1\widehat{X}(t),\delta^1Y(t),\delta^1Z(t)-p^0(t)\delta\sigma(t,v(t))\mathbf{1}_{E_\varepsilon}(t)]^\intercal\\
  &\quad-\frac{1}{2}[1, p^0(t),p^0(t)\sigma_x(t)+q^0(t)]
                         D_{xyz}^2f(t) [1, p^0(t),p^0(t)\sigma_x(t)+q^0(t)]^\intercal(X^{1,\varepsilon}(t))^2\Big)^2\bigg]\\
&=\frac{1}{4}\mathbb{E}\bigg[\Big(\int_0^T I_1(t)+ I_2(t) +I_3(t)+I_4(t)dt\Big)^2\bigg],\\
\end{aligned}
$$
where
$$
\begin{aligned}
I_1(t)&:=
[\delta^1X(t), \delta^1\widehat{X}(t),\delta^1Y(t),\delta^1Z(t)-p^0(t)\delta\sigma(t,v(t))\mathbf{1}_{E_\varepsilon}(t)]
  D^2f^{\tilde{\rho}\rho\varepsilon}(t)\\
  &\qq\qq
  [\delta^1X(t), \delta^1\widehat{X}(t),\delta^1Y(t),\delta^1Z(t)-p^0(t)\delta\sigma(t,v(t))\mathbf{1}_{E_\varepsilon}(t)]^\intercal\\
  &\q\ -[\delta^1X(t), \delta^1Y(t),\delta^1Z(t)-p^0(t)\delta\sigma(t,v(t))\mathbf{1}_{E_\varepsilon}(t)]
  D^2_{xyz}f^{\tilde{\rho}\rho\varepsilon}(t)\\
 &\qq\qq
 [\delta^1X(t), \delta^1Y(t),\delta^1Z(t)-p^0(t)\delta\sigma(t,v(t))\mathbf{1}_{E_\varepsilon}(t)]^\intercal ,\\
 I_2(t)&:=[\delta^1X(t), \delta^1Y(t),\delta^1Z(t)-p^0(t)\delta\sigma(t,v(t))\mathbf{1}_{E_\varepsilon}(t)]
  D^2_{xyz}f^{\tilde{\rho}\rho\varepsilon}(t)\\
 &\qq\qq
 [\delta^1X(t), \delta^1Y(t),\delta^1Z(t)-p^0(t)\delta\sigma(t,v(t))\mathbf{1}_{E_\varepsilon}(t)]^\intercal \\
  &\q
  -[X^{1,\varepsilon}(t), Y^{1,\varepsilon}(t),Z^{1,\varepsilon}(t)-p^0(t)\delta\sigma(t,v(t))\mathbf{1}_{E_\varepsilon}(t)]
  D^2_{xyz}f^{\tilde{\rho}\rho\varepsilon}(t)  \\
 &\qq\qq
 [X^{1,\varepsilon}(t),Y^{1,\varepsilon}(t),Z^{1,\varepsilon}(t)
 -p^0(t)\delta\sigma(t,v(t))\mathbf{1}_{E_\varepsilon}(t)]^\intercal,\\
 I_3(t)&:=[X^{1,\varepsilon}(t), Y^{1,\varepsilon}(t),Z^{1,\varepsilon}(t)-p^0(t)\delta\sigma(t,v(t))\mathbf{1}_{E_\varepsilon}(t)]
  D^2_{xyz}f^{\tilde{\rho}\rho\varepsilon}(t)  \\
 &\qq\qq
 [X^{1,\varepsilon}(t),Y^{1,\varepsilon}(t),Z^{1,\varepsilon}(t)-p^0(t)\delta\sigma(t,v(t))\mathbf{1}_{E_\varepsilon}(t)]^\intercal\\
 &\q\
  -[1, p^0(t),p^0(t)\sigma_x(t)+q^0(t)]
                         D_{xyz}^2f^{\tilde{\rho}\rho\varepsilon}(t) [1, p^0(t),p^0(t)\sigma_x(t)+q^0(t)]^\intercal(X^{1,\varepsilon}(t))^2,\\
  I_4(t)&:=
  [1, p^0(t),p^0(t)\sigma_x(t)+q^0(t)]
                         (D_{xyz}^2f^{\tilde{\rho}\rho\varepsilon}(t)-D_{xyz}^2f(t))
                          [1, p^0(t),p^0(t)\sigma_x(t)+q^0(t)]^\intercal(X^{1,\varepsilon}(t))^2
   \end{aligned}
$$
and
$$
D^2_{xyz}f^{\tilde{\rho}\rho\varepsilon}(t)=2\int_0^1\int_0^1\rho
D^2_{xyz}f(t,\bar{\Gamma}(t)+\tilde{\rho}\rho(\Gamma^\varepsilon(t)-\bar{\Gamma}(t)),v^\varepsilon(t),\alpha(t-))d\tilde{\rho}d\rho.
$$

\ms

We analyse them one by one.

\ms

$\mathrm{c)}$\ For $I_1(t)$, from the boundness of the second-order derivatives of $f$ with respect to $(x,x',y,z)$ we have
\begin{align} \nonumber
\begin{aligned}
|I_1(t)|&\leq C\bigg\{|\delta^1 \widehat{X}(t)|\cdot |\delta^1 X(t)|+|\delta^1 \widehat{X}(t)|^2+|\delta^1 \widehat{X}(t)|\cdot |\delta^1 Y(t)|\\
        &\quad+|\delta^1 \widehat{X}(t)|   |\delta^1Z(t)-p^0(t)\delta\sigma(t,v(t))\mathbf{1}_{E_\varepsilon}(t)|\bigg\}.
\end{aligned}
\end{align}
Hence, it follows from \autoref{le 4.1}, \autoref{pro 4.2}, \autoref{le 4.3} and \autoref{pro 4.5}
\begin{align}\nonumber
\begin{aligned}
&\mathbb{E}\bigg[\Big(\int_0^T |I_1(t)|dt\Big)^2\bigg]\\
&\leq C\mathbb{E}\bigg[\Big(\int_0^T |\delta^1 \widehat{X}(t)|\cdot (|\delta^1 X(t)|+|\delta^1 \widehat{X}(t)|+|\delta^1 Y(t)|+  |\delta^1Z(t)-p^0(t)\delta\sigma(t,v(t))\mathbf{1}_{E_\varepsilon}(t)|)dt\Big)^2\bigg]\\
&\leq C\dbE\bigg[\sup_{t\in[0,T]} |\delta^1 \widehat{X}(t)|^2\cdot
\( \sup_{t\in[0,T]}|\delta^1 X(t)|^2+\sup_{t\in[0,T]}|\delta^1 \widehat{X}(t)|^2+\sup_{t\in[0,T]}|\delta^1 Y(t)|^2 \\
&\qq \qq\qq\qq\qq\qq\qq\q
+(\int_0^T|\delta^2Z(t)+Z^{1,\varepsilon}(t)-p^0(t)\delta\sigma(t,v(t))\mathbf{1}_{E_\varepsilon}(t)|dt)^2\)  \bigg] \\
&\leq C \dbE\bigg[\sup_{t\in[0,T]} |\delta^1 \widehat{X}(t)|^2\cdot
\( \sup_{t\in[0,T]}|\delta^1 X(t)|^2+\sup_{t\in[0,T]}|\delta^1 \widehat{X}(t)|^2+\sup_{t\in[0,T]}|\delta^1 Y(t)|^2 \\
&\q
+\int_0^T|\delta^2Z(t)|^2dt
                  +\sup_{t\in[0,T]}|X^{1,\varepsilon}(t)|^2\cd\int_0^T|p^0(t)\sigma_x(t)+q^0(t)|^2dt\\
                &   \q+\sup_{t\in[0,T]}|\widehat{X}^{1,\varepsilon}(t)|^2\cd\int_0^T|p^0(t)\sigma_{x'}(t)+q^1(t)|^2dt\)  \bigg]\leq \e^3.
\end{aligned}
\end{align}

\ms

$\mathrm{d)}$\ As for $I_2(t)$, thanks to the boundness of the derivatives of $f$ with respect to $(x,y,z)$, \autoref{le 4.1}, \autoref{pro 4.2}, \autoref{le 4.3} and \autoref{pro 4.5} we have
\begin{align}\nonumber
\begin{aligned}
&\mathbb{E}\bigg[\Big(\int_0^T |I_2(t)|dt\Big)^2\bigg]\\
&\leq C\dbE\bigg[\Big(\int_0^T \Big|\(\d^2X(t)+\d^2Y(t)+\d^2Z(t) \)\(\d^1X(t)+\d^1Y(t)+\delta^1Z(t)-p^0(t)\delta\sigma(t,v(t))
\mathbf{1}_{E_\varepsilon}(t)\\
&\qq\qq\qq\qq\qq\qq\qq\qq\qq\qq\qq\qq
+X^{1,\e}(t)+Y^{1,\e}(t)+Z^{1,\e}(t)\)\Big|dt\Big)^2\bigg]\\
&\leq C\Bigg\{\dbE\[\sup_{t\in[0,T]}|\d^2X(t)|^4+\sup_{t\in[0,T]}|\d^2Y(t)|^4+\(\int_0^T|Z(t)|^2dt\)^2\]\Bigg\}^\frac{1}{2}\cd\\
&\qq
\Bigg\{\dbE\[\sup_{t\in[0,T]}|\d^1X(t)+X^{1,\e}(t)|^4+\sup_{t\in[0,T]}|\d^1Y(t)+Y^{1,\e}(t)|^4\\
&\qq
+\(\int_0^T|\delta^1Z(t)|^2+|[p^0(t)\si_x(t)+q^0(t)]X^{1,\e}(t)+[p^0(t)\si_{x'}(t)+q^1(t)]\h{X}^{1,\e}(t)|^2dt\)^2\]\Bigg\}^\frac{1}{2}\\
&\leq C\e^3.
\end{aligned}
\end{align}

\ms

$\mathrm{e)}$\  Next, we analyse $I_3(t)$. Thanks to the boundness of $\sigma_x,\sigma_{x'}$  and $D^2_{xyz}f^{\tilde{\rho}\rho\varepsilon}$, H\"{o}lder
inequality, \autoref{pro 4.2}, we have
\begin{align}\nonumber
\begin{aligned}
&\mathbb{E}\bigg[\Big(\int_0^T |I_3(t)|dt\Big)^2\bigg]\\
&\leq C \mathbb{E}\bigg[\bigg(\int_0^T \Big| X^{1,\varepsilon}(t)\widehat{X}^{1,\varepsilon}(t)
\[(p_0(t)+1)(p_0(t)\si_{x'}(t)+p_1(t)+q_1(t))+p_0(t)q_1(t)\si_{x}(t)\\
&\qq+(p_0(t)\si_{x'}(t)+p_1(t))(p_0(t)\si_{x}(t)+q_0(t))+q^0(t)q^1(t)\]\\
&\qq+(\widehat{X}^{1,\varepsilon}(t))^2\[(p_1(t))^2+(p_0(t)\si_{x'}(t)+p_1(t))(p_0(t)\si_{x'}(t)
+q_1(t))+p_0(t)q_1(t)\si_{x}(t)+(q^1(t))^2\]\Big|dt\bigg)^2\bigg].
\end{aligned}
\end{align}
 We just estimate the term $\mathbb{E}\bigg[\(\int_0^T| X^{1,\varepsilon}(t)\widehat{X}^{1,\varepsilon}(t)q^0(t)q^1(t)dt\)^2\bigg]$, since
 the other terms  are analogous.
\begin{align}\nonumber
\begin{aligned}
&\mathbb{E}\bigg[\(\int_0^T| X^{1,\varepsilon}(t)\widehat{X}^{1,\varepsilon}(t)q^0(t)q^1(t)dt\)^2\bigg]\\
&\leq \bigg\{\mathbb{E}\Big[\sup\limits_{t\in[0,T]}|\widehat{X}^{1,\varepsilon}(t)|^8\Big]\bigg\}^\frac{1}{4}
\bigg\{\mathbb{E}\Big[\sup\limits_{t\in[0,T]}|X^{1,\varepsilon}(t)|^8\Big]\bigg\}^\frac{1}{4}
\bigg\{\mathbb{E}\Big[\Big(\int_0^T| q^0(t)|^2dt \Big)^4\Big]\bigg\}^\frac{1}{4}
\bigg\{\mathbb{E}\Big[\Big(\int_0^T| q^1(t)|^2dt \Big)^4\Big]\bigg\}^\frac{1}{4}\\
& \leq\varepsilon^3
\end{aligned}
\end{align}
Consequently, one has $\mathbb{E}\bigg[\Big(\int_0^T |I_3(t)|dt\Big)^2\bigg] \leq \varepsilon^3.$

\ms

$\mathrm{f)}$\ From the continuity of the second-order derivatives of $f$ with respect to $v$,
Dominated Convergence Theorem allows to show
$\mathbb{E}\bigg[\Big(\int_0^T |I_4(t)|dt\Big)^2\bigg]\leq \varepsilon^2\rho(\varepsilon).$
Hence, we have
\begin{align}\nonumber
\begin{aligned}
&\mathbb{E}\bigg[\Big(\int_0^T
\frac{1}{2} [\delta^1X(t), \delta^1\widehat{X}(t),\delta^1Y(t),\delta^1Z(t)-p^0(t)\delta\sigma(t,v(t))\mathbf{1}_{E_\varepsilon}(t)]
  D^2f^{\tilde{\rho}\rho\varepsilon}(t)\\
  &\qq\qq\qq
  [\delta^1X(t), \delta^1\widehat{X}(t),\delta^1Y(t),\delta^1Z(t)-p^0(t)\delta\sigma(t,v(t))\mathbf{1}_{E_\varepsilon}(t)]^\intercal\\
  &\quad-\frac{1}{2}[1, p^0(t),p^0(t)\sigma_x(t)+q^0(t)]
                         D_{xyz}^2f(t) [1, p^0(t),p^0(t)\sigma_x(t)+q^0(t)]^\intercal(X^{1,\varepsilon}(t))^2dt\Big)^2\bigg]
\leq \varepsilon^2\rho(\varepsilon ).
\end{aligned}
\end{align}

\end{document}